\documentclass[12pt,a4paper]{article}

\usepackage[latin1]{inputenc}
\usepackage[english]{babel}
\usepackage[T1]{fontenc}
\usepackage{amsmath}
\usepackage{commath} %for \norm (Mic)
\usepackage{enumerate} %for enumerate (Mic)
\usepackage{amsfonts}
\usepackage{amssymb}
\usepackage{amsthm}
\usepackage{ dsfont }

\usepackage{ stmaryrd }
\usepackage{tikz}
\usetikzlibrary{matrix,arrows.meta}
\usepackage[title]{appendix}

\allowdisplaybreaks

\usepackage{todonotes}
\usepackage{hyperref}
\usepackage{cleveref} %for clever referencing (Mic)
\usepackage{comment}

\usepackage{xcolor} %for color
\usepackage[a4paper,top=2.1cm,bottom=2.60cm,left=2.5cm,right=2.5cm]{geometry} %toglierlo

\theoremstyle{plain}
\newtheorem{theorem}{Theorem}[section]
\newtheorem{lemma}[theorem]{Lemma}
\newtheorem{proposition}[theorem]{Proposition}

\theoremstyle{definition} 
\newtheorem{remark}[theorem]{Remark}
\newtheorem{definition}[theorem]{Definition}

\newcommand{\xn}{X_{\rm \scriptscriptstyle N}}

\begin{document}

\author{ Paolo Luzzini \thanks{Dipartimento di Matematica `Tullio Levi-Civita', Universit\`a degli Studi di Padova, Via Trieste 63, Padova 35121, Italy.} , Michele Zaccaron \thanks{Dipartimento di Matematica `Tullio Levi-Civita', Universit\`a degli Studi di Padova, Via Trieste 63, Padova 35121, Italy.}}
\title{A few results on permittivity variations   in electromagnetic cavities}
\maketitle

\abstract{We study the eigenvalues of  time-harmonic Maxwell's equations in a cavity upon 
changes in the electric permittivity  $\varepsilon$ of the medium. We prove that all the eigenvalues, both simple and multiple, are locally Lipschitz continuous with respect to $\varepsilon$.  Next, we show that simple eigenvalues and the symmetric functions of multiple eigenvalues depend real analytically upon $\varepsilon$ and we provide an explicit formula for their 
derivative in $\varepsilon$.  As an application of these results, we show that for a generic permittivity  all the Maxwell eigenvalues are simple.}

\vspace{.4cm}

\vspace{.4cm}

\section{Introduction}

In this paper we present some sensitivity results for the eigenvalues of the time-harmonic Maxwell's equations in a cavity  upon perturbations of the permittivity parameter.
The cavity is represented by a bounded  domain (i.e. a connected open subset) $\Omega $ of $\mathbb{R}^3$, and it is thought made of a  material which in general is inhomogeneous and anisotropic. Accordingly, the   permittivity $\varepsilon$ of the medium filling the domain $\Omega$ is described by a $(3 \times 3)$-matrix valued function in $\Omega$. In particular cases where the material presents additional properties and symmetries, the permittivity $\varepsilon$ takes simpler forms, for example becoming scalar in the case of a isotropic material, or even a scalar constant if the medium is both isotropic and homogeneous.

The eigenfrequency problem in a bounded domain $\Omega \subset \mathbb{R}^3$ consists in finding  two non-zero eigenfields $E,H$ and a non-zero eigenfrequency $\omega$ (also called angular frequency) such that the time-harmonic Maxwell's equations are satisfied in $\Omega$, namely
\begin{equation}\label{sys:max}
\mathrm{curl} E -{\rm i}\,\omega \mu H =0, \qquad \mathrm{curl} H +{\rm i}\,\omega \varepsilon E=0 \qquad \mbox{ in } \Omega.
\end{equation}
The vector field $E$ denotes the electric part of the electromagnetic field, while $H$ the magnetic one.
Furthermore, $\varepsilon$ and $\mu$ are $(3 \times 3)$-matrix valued maps which represent the electric permittivity  and the magnetic permeability tensors of the medium, respectively.
For the sake of simplicity, since we are interested in studying the behavior of problem \eqref{sys:max}  upon  variation of $\varepsilon$,  we normalize the permeability to have $\mu =I_3$, where $I_3$ denotes the $(3 \times 3)$-identity matrix.  By applying the $\mathrm{curl}$ operator to  the first equation of \eqref{sys:max} and setting $\omega^2 = \lambda$ we obtain 
\begin{equation*}
\operatorname{curl}\operatorname{curl} E = \lambda \varepsilon E\qquad \mbox{ in } \Omega.
\end{equation*}
Since the divergence of a $\mathrm{curl}$ is always zero then
\[
\mathrm{div}\, \varepsilon E = 0 \qquad \mbox{ in } \Omega.
\]
We couple the system with the boundary conditions of a perfect conductor, which for the electric field $E$ read as follows:
\begin{equation} \label{introluz:elec:bc}
\nu \times E = 0 \qquad \mbox{ on } \partial \Omega.
\end{equation}
Here $\nu$ denotes the outer unit normal vector to the boundary of $\Omega$ hence  condition \eqref{introluz:elec:bc} means that the electric field is orthogonal to the surface $\partial \Omega$.
Therefore, we   arrive at the following (electric) eigenvalue problem:
\begin{equation}\label{prob:eig}
\begin{cases}
\operatorname{curl}\operatorname{curl} E = \lambda \, \varepsilon E, \qquad &\mbox{ in } \Omega,\\
\mathrm{div}\, \varepsilon E = 0, & \mbox{ in } \Omega,\\
\nu \times E = 0, \qquad &\mbox{ on } \partial \Omega.
\end{cases}
\end{equation}
Note that it is also possible to obtain  the magnetic counterpart of problem \eqref{prob:eig}.
However, in the present work, we will devote our attention to the electric side.
The spectrum of problem  \eqref{prob:eig} is discrete (cf. \cite[Thm. 4.34]{kihe}) and  it consists of   a divergent sequence of  $\varepsilon$-dependent non-negative eigenvalues $\{\lambda_j[\varepsilon]\}_{j \in \mathbb{N}}$ of finite multiplicity that can be arranged in an increasing way
\[
0\leq \lambda_1[\varepsilon] \leq \lambda_2[\varepsilon] \leq \cdots \leq \lambda_n[\varepsilon] \leq \cdots \nearrow +\infty,
\]
where each eigenvalue is repeated in accordance with its multiplicity.
The kernel $K^\varepsilon(\Omega)$ of problem \eqref{prob:eig}, i.e. those eigenfields associated with $\lambda =0$, is composed of curl-free vector fields which are normal to the boundary and such that $\operatorname{div}\varepsilon E=0$ in $\Omega$, namely
\begin{equation} \label{zero:eigenspace:eps}
K^\varepsilon(\Omega) = \set{E \in L^2(\Omega)^3 : \operatorname{curl}E=0 \text{ in }\Omega,\, \operatorname{div}\varepsilon E =0\text{ in }\Omega,\, \nu \times E=0 \text{ on }\partial\Omega}.
\end{equation}
If $m \in \mathbb{N}$ is the number of connected components of the boundary of $\Omega$, then $\mathrm{dim}_\mathbb{R}K^\varepsilon(\Omega)=m-1$. In particular, if $\partial\Omega$ has only one connected component, $K^\varepsilon(\Omega)=\{0\}$. It is worth noting that the presence of the zero eigenvalue, and its multiplicity, depends only on the topology of $\Omega$. For a proof we refer to \cite[Prop. 6.1.1]{AsCiLa18} (see also \cite[Prop. 3.18]{abdg}, \cite[Ch. IX-A \S 1.3]{DaLi90}).

%Since $\lambda=\omega^2$ where $\omega>0$ is the eigenfrequency of system \eqref{sys:max}, we are only interested
% in the positive eigenvalues, and thus in our analysis we will forget about the zero eigenvalue even if  the kernal is nontrivial.
%

On the one hand, the aim of our work is to understand the dependence of all the eigenvalues $\lambda_j[\varepsilon]$, both simple 
and multiple, with respect to variations of the permittivity $\varepsilon$. On the other hand, as a  consequence of  our sensitivity analysis, we   prove that  all the eigenvalues are generically simple with respect to $\varepsilon$.

The mathematical study of Maxwell's equations and in particular of electromagnetic cavities has great  interest not only from the theoretic side but also for its real world  applications, for example in designing cavity resonators or shielding structures for electronic circuits. Here we mention, without the sake of completeness, the monographs 
\cite{Ce96, DaLi90, GiRa86, Mo03, Ne01,RoStYa12} and the classical papers \cite{Co90, Co91, CoDo99} for a complete introduction to this field and a detailed discussion 
of both theoretic and applied problems in the mathematical theory of electromagnetism. For more recent papers we refer to, e.g., \cite{AmBaWo00, BaPaSc16, CoCoMo19, LaSt19, Pa17, LaZa21}. Incidentally, we note that  in \cite{LaZa20}  Lamberti and the second named author  have considered the eigenvalues of problem \eqref{prob:eig} with fixed and constant permittivity $\varepsilon = I_3$ on a variable domain and proved  a real analytic dependence  upon variation of the shape of the domain.
To the best of the authors' knowledge, the dependence of the eigenvalues $\lambda_j[\varepsilon]$ upon perturbation of $\varepsilon$ has not yet been investigated.

As a first step,  we  consider the stability of the eigenvalues and we show that all the eigenvalues, both simple and multiple, are continuous with respect to $\varepsilon$ varying in $W^{1,\infty}$. Actually we are able to prove a stronger result, indeed we show that the eigenvalues are locally Lipschitz continuous 
in $\varepsilon$ (see Theorem \ref{thm:loclip}). 

Then, we pass to consider higher regularity properties. At this stage we face a first issue  related to  bifurcation phenomena of multiple eigenvalues which is  common to any eigenvalue problem depending on a parameter. 
In our case, if one has a multiple eigenvalue $\lambda= \lambda_{j}[\varepsilon]=\cdots = \lambda_{j+m-1}[\varepsilon]$
   and $\varepsilon$ is slightly perturbed, $\lambda$ could split into different eigenvalues of lower multiplicity and thus the corresponding branches can  present a corner at the splitting point, and then be not differentiable. A possible way to overturn this problem is to consider only perturbations  $\{\varepsilon_t\}_{t \in \mathbb{R}}$,
    with $\varepsilon_0=\varepsilon$, depending on a single scalar parameter $t \in \mathbb{R}$ and consider  the one-sided derivative of the multiple eigenvalue at $t=0$ (see, e.g., \cite{DjFaWe21,FaWe18} for  the one-sided shape derivative of a--possibly multiple--eigenvalue for two different problems). Note that this approach, being based on  the variational characterization of the eigenvalues,  has been effectively applied only to the first non-zero eigenvalue.
    
     Here we adopt a different point of view that allows us to deal with multiple eigenvalues and general (infinite dimensional) perturbations of the permittivity:  instead of considering a single eigenvalue we consider  the symmetric functions of multiple eigenvalues and we show that they depend real analytically on 
   $\varepsilon$. In addition we provide an explicit formula for the (Fr\'echet) derivative in $\varepsilon$ of the symmetric functions of the eigenvalues (see Theorem \ref{thm:diffeps}).  This suggests that the symmetric functions are  a natural quantity to consider when dealing with the regularity (and the optimization) of multiple eigenvalues. This approach was introduced by Lamberti and Lanza de Cristoforis \cite{LaLa04} and later adopted in other works (see, e.g., \cite{BuLa13, BuLa15, LaZa20, LaLuMu21, LaMuTa21}).  We also consider the case of perturbations depending on a single scalar parameter, like the ones we introduced above, and we prove a Rellich-Nagy-type theorem which describes the bifurcation phenomenon of multiple eigenvalues. More precisely, we show that all the eigenvalues splitting from a multiple eigenvalue of multiplicity $m$ can be described by $m$ real analytic functions of the scalar parameter  (see Theorem \ref{thm:RN}). 
 
As an application of the above described results, we show that all the non-zero eigenvalues of problem \eqref{prob:eig} are simple for a 
generic permittivity. That is, in few words, given any permittivity $\varepsilon$ it is always possible to find a perturbation 
$\tilde \varepsilon$ as close as desired to $\varepsilon$ such that the non-zero eigenvalues 
$\{\lambda_j[\tilde \varepsilon]\}_{j \in \mathbb{N}}$ are all simple (see Theorem \ref{thm:gensim}).

To a certain extent, our work is inspired by Lamberti \cite{La09} and Lamberti and Provenzano \cite{LaPr13}  where the authors investigate the behavior of the eigenvalues of the Laplacian and of a general elliptic operators upon perturbations of mass density.
Incidentally, we mention that this paper is a first step towards  understanding  the behavior of Maxwell eigenvalues upon permittivities variations. In particular the authors plan to investigate issues related to the optimization of Maxwell eigenvalues with respect to $\varepsilon$ in a future work. 

 After the present introduction the paper is organized as follows: Section \ref{sec:pre} is a section of preliminaries containing  notation, the functional setting  and basic results about  the eigenvalue problem. In Section \ref{sec:cont} we prove that all the eigenvalues are locally Lipschitz continuous in $\varepsilon$. In Section \ref{sec:an} we show that the symmetric functions of the eigenvalues depend analytically upon $\varepsilon$ and we obtain an explicit formula for the $\varepsilon$-derivative. Moreover, we prove a Rellich-Nagy-type result for permittivities depending on a single scalar parameter. Finally, in Section \ref{sec:gen} we show that all the non-zero eigenvalues are simple for a generic permittivity.
 
\section{Some preliminaries}\label{sec:pre}
%Given any matrix (or vector) $M \in \mathbb{R}^{m \times n}$, with $M^T \in \mathbb{R}^{n \times m}$ we denote its transpose. 
If $\mathcal{X}$ is a Hilbert space of scalar functions, by $\mathcal{X}^3$ we denote the Hilbert space of vector-valued functions whose components belong to $\mathcal{X}$, endowed with the natural inner product

\begin{equation*}
\langle f,g \rangle_{\mathcal{X}^3} = \sum_{i=1}^3 \langle f_i,g_i\ \rangle_\mathcal{X}   
\end{equation*}
for all $f=(f_1,f_2,f_3)$, $ g=(g_1,g_2,g_3)
 \in \mathcal{X}^3$, 
where $\langle \cdot, \cdot \rangle_{\mathcal{X}}$ is the inner product of $\mathcal{X}$.
%By $L^2(\Omega)$ we denote the Lebesgue space of square-summable real 
%valued functions endowed with its standard scalar product. 
In this sense,  e.g.,   if $L^2(\Omega)$ is the standard Lebesgue space of square integrable real valued functions, then the space $L^2(\Omega)^3$ is endowed with the inner product
\[
\int_{\Omega} u \cdot v \,dx  = \int_{\Omega} (u_1v_1+u_2v_2+u_3v_3) \,dx \qquad \forall u,v \in L^2(\Omega)^3.
\]

Let $\Omega$ be a bounded domain of  $\mathbb{R}^3$.
% \textcolor{red}{that either $\Omega \subset \mathbb{R}^3$ is a bounded set of class $\mathcal{C}^{1,1}$ or that it is Lipschitz regular and convex. Ma ci serve veramente di lavorare con roba $H^1$ (e dunque la Gaffney)? Nota anche che qua si usano vettori colonna...cambiare???e che per denotare i prodotti scalari ho usato sempre $(,)$ tranne che nel caso di $L^2_\varepsilon$. Infine, serve $\Omega$ o $\Omega$?}. 
We denote by $L^\infty(\Omega)^{3 \times 3}$ and  $W^{1,\infty}(\Omega)^{3 \times 3}$ the spaces of real matrix-valued functions $M=\left( M_{ij} \right)_{1 \leq i,j \leq 3}:\Omega \to   \mathbb{R}^{3 \times 3}$ 
whose components are in $L^\infty(\Omega)$ and $W^{1,\infty}(\Omega)$, respectively. We endow these spaces   with the norms
\begin{equation*}
\norm{M}_{L^\infty(\Omega)^{3 \times 3}} := \max_{1 \leq i,j \leq 3} \norm{M_{ij}}_{L^\infty(\Omega)}
\end{equation*}
and 
\begin{equation*}
\norm{M}_{W^{1,\infty}(\Omega)^{3 \times 3}} := \max_{1 \leq i,j \leq 3} \norm{M_{ij}}_{W^{1,\infty}(\Omega)}.
\end{equation*}
For the sake of simplicity, we will respectively  write $L^\infty(\Omega)$ and $W^{1,\infty}(\Omega)$ instead of $L^\infty(\Omega)^{3 \times 3}$ and $W^{1,\infty}(\Omega)^{3 \times 3}$, and the space we are referring to will be clear from the context.
%\textcolor{blue}{``For the sake of simplicity we will often write $L^\infty(\Omega)$ instead of $L^\infty(\Omega)^{3\times 3}$ and similar, the case whether we are dealing with scalars, numbers or matrices will be understood in the text.
%we will also denote with $|\xi|:= \xi \cdot \xi$ the norm of a vector and $|M| = \max_{1 \leq i,j \leq 3} |M_{ij}|$. the norm of a matrix (they are all equivalent)"}
Let $M \in L^\infty(\Omega)$. One has the following trivial inequalities that we will exploit in the paper: 
\begin{equation*}
\abs{M \xi \cdot \xi} \leq 3 \norm{M}_{L^\infty(\Omega)} \abs{\xi}^2, \qquad \abs{M \xi} \leq 3 \norm{M}_{L^\infty(\Omega)} |\xi|
\end{equation*}
%full inequality con conti 
%\[
%\abs{M \xi \cdot \xi} \leq \norm{M}_{L^\infty(\Omega)} \sum_{i,j=1}^3 |\xi_i \xi_j| \leq \norm{M}_{L^\infty(\Omega)} \sum_{i,j=1}^3 \frac{1}{2}(\xi_i^2 + \xi_j^2) = 3 \norm{M}_{L^\infty(\Omega)} \abs{\xi}^2
%\]
for all $\xi \in \mathbb{R}^3$ and a.e. in $\Omega$.  
%\textcolor{blue}{cambiare le xi con zeta}

In order to consider our eigenvalue problem, we first need to specify where we take the permittivities $\varepsilon$.
From now on we will assume that:
\begin{equation}\label{Omega_def}
\text{$\Omega$ is a bounded domain of $\mathbb{R}^3$ of class $C^{1,1}$.}
\end{equation}
The admissible set where we take the permittivities is the following
\begin{align*}
\mathcal{E}:= \Big\{\varepsilon \in W^{1,\infty} \left(\Omega\right) \cap  \,\, &\mathrm{Sym}_3 (\Omega) : \\
&\exists \, c>0 \text{ s.t. }  \varepsilon(x) \, \xi \cdot \xi \geq c \, \abs{\xi}^2 \text{ for a.a. } x \in  \Omega, \text{ for all }\xi \in \mathbb{R}^3\Big\},
\end{align*}
where $\mathrm{Sym}_3 (\Omega)$ denotes the set of $(3 \times 3)$-symmetric matrix valued functions in $\Omega$.
Given $\varepsilon \in \mathcal{E}$, we denote by $c_\varepsilon >0$  the greatest positive constant  that guarantees the coercivity condition in the above definition, that is
\begin{equation}\label{def:cc}
c_\varepsilon := \max\Big\{c>0 :   \varepsilon(x) \, \xi \cdot \xi \geq c \, \abs{\xi}^2 \text{ for a.a. } x \in  \Omega, \text{ for all }\xi \in \mathbb{R}^3\Big\}.
\end{equation}
The set $\mathcal{E}$ is  open in  $W^{1,\infty}(\Omega)\cap  \,\, \mathrm{Sym}_3 (\Omega)$. This is implied by the continuity of the map 
\begin{equation*}
\left( \mathcal{E}, \|\cdot\|_{L^{\infty}(\Omega)} \right) \to \mathbb{R}_+, \qquad \varepsilon \mapsto c_\varepsilon.
\end{equation*}
Indeed let $\varepsilon_1, \varepsilon_2 \in \mathcal{E}$. Since $\abs{(\varepsilon_2 - \varepsilon_1) \, \xi \cdot \xi} \leq 3 \norm{\varepsilon_2 - \varepsilon_1}_{L^\infty(\Omega)} \abs{\xi}^2$ a.e. in $\Omega$, then
\begin{equation*} 
\varepsilon_2 \, \xi \cdot \xi = \varepsilon_1 \, \xi \cdot \xi + (\varepsilon_2 - \varepsilon_1) \, \xi \cdot \xi \geq \left(c_{\varepsilon_1} - 3 \norm{\varepsilon_2 - \varepsilon_1}_{L^\infty(\Omega)} \right) \abs{\xi}^2.
\end{equation*}
Hence $c_{\varepsilon_2} \geq c_{\varepsilon_1} - 3 \|\varepsilon_2 - \varepsilon_1\|_{L^\infty(\Omega)}$. 
Eventually exchanging the role of $\varepsilon_1$ and   $\varepsilon_2$, we have that 
\begin{equation}\label{contcoerc}
\abs{c_{\varepsilon_2} - c_{\varepsilon_1}} \leq 3 \norm{\varepsilon_2 - \varepsilon_1}_{L^\infty(\Omega)}.
\end{equation}

%lower semicontinuity of the map elliptic constant
\begin{comment}
We will actually show more, namely that if $\varepsilon_k \xrightarrow[k \to \infty]{L^\infty} \varepsilon$, then $\lim\inf_{k\to \infty} c_k \geq c$.
Suppose by contradiction that there exists $0<\delta < c/2$ such that for all $n \in \mathbb{N}$ there exists a $k_n = k \geq n$ such that $c_k \leq c-2 \delta$. Then, considering the sequence $\{c_{k} \}_{k \in \mathbb{N}}$, we would have that for every $k \in \mathbb{N}$ there exists a $0 \neq \tilde{\xi}_k = \tilde{\xi} \in \mathbb{R}^3$ such that 
\begin{equation}
c_m |\tilde{\xi}|^2  \leq \varepsilon_k \, \tilde{\xi} \cdot \tilde{\xi} \leq (c-\delta) |\tilde{\xi}|^2.
\end{equation}
But then, since $\abs{(\varepsilon_k -\varepsilon) \, \tilde{\xi} \cdot \tilde{\xi}} \leq \norm{\varepsilon_k - \varepsilon}_{L^\infty(\Omega)} |\tilde{\xi}|^2 \leq \delta/2 |\tilde{\xi}|^2$ for $k$ sufficiently large, we have that
\begin{equation}
(c - \delta/2) |\tilde{\xi}|^2 \leq \varepsilon \, \tilde{\xi} \cdot \tilde{\xi} + (\varepsilon_k - \varepsilon) \, \tilde{\xi} \cdot \tilde{\xi} = \varepsilon_k \, \tilde{\xi} \cdot \tilde{\xi} \leq (c- \delta) |\tilde{\xi}|^2,
\end{equation}
which in turn implies that $\delta > 2 \delta$, a contradiction.
\end{comment}

Let $\varepsilon \in \mathcal{E}$. We denote by 
$L^2_\varepsilon(\Omega)$ 
 the space $L^2(\Omega)^3$ endowed with the inner product
\begin{equation} \label{eps:inner:product}
\langle u, v \rangle_\varepsilon = J_\varepsilon[u][v]:= \int_\Omega \varepsilon u \cdot v \, dx \qquad \forall u,v \in L^2(\Omega)^3.
\end{equation}
Note that the above inner product induces a norm equivalent to the standard $L^2$-norm since
\begin{equation*}
c_{\varepsilon} \int_\Omega \abs{u}^2 dx  \leq \int_\Omega \varepsilon u  \cdot u \, dx  \leq  3 \norm{\varepsilon}_{L^\infty(\Omega)} \int_\Omega \abs{u}^2 dx \qquad \forall u \in  L^2(\Omega)^3. 
\end{equation*}

Next we introduce the natural functional setting and tools in order to deal with problem \eqref{prob:eig}.  
By $H(\mathrm{curl}, \Omega)$  we denote the  space of vector fields  $ u \in L^2(\Omega)^3$ with distributional $\mathrm{curl}$ in $L^2(\Omega)^3$, i.e. those square integrable vector fields for which  there exists a function $\mathrm{curl}\, u \in  L^2(\Omega)^3$ such that 
\begin{equation} \label{distrib:eps:div:dfn}
\int_\Omega u \cdot \operatorname{curl}\varphi \, dx = \int_\Omega \operatorname{curl} u \cdot \varphi \, dx 
\qquad \forall \varphi \in C^\infty_c(\Omega)^3.
\end{equation}
We endow this space with the  inner product
\begin{equation*}
\langle u,v\rangle_{H(\operatorname{curl}, \Omega)} := \int_\Omega \varepsilon u \cdot v \, dx + \int_\Omega \operatorname{curl}u \cdot \operatorname{curl}v \, dx \qquad \forall u,v \in H(\operatorname{curl},\Omega),
\end{equation*}
which makes it a Hilbert space.
By $H_0(\mathrm{curl}, \Omega)$ we denote  the closure of $C^\infty_c(\Omega)^3$ in $H(\mathrm{curl}, \Omega)$.
If a vector field $u$ is regular enough to be traced on the boundary, say it is smooth up to the boundary, then the \emph{tangential trace} of $u$ coincides exactly with the cross product between its restriction to $\partial \Omega$ and the outer unit normal, i.e. $\nu \times u\rvert_{\partial \Omega}$. From now on we use the same notation also to denote the tangential trace of a vector field $u \in H(\operatorname{curl}, \Omega)$, which in general is just an element of the dual space of $H^{1/2}(\partial \Omega)^3$ (see \cite[Thm. 2.11]{GiRa86}).
%Observe moreover that given a vector field $\varphi \in H^1(\Omega)^3$ the standard (curl-type) Green's formula
%\begin{equation*}
%\int_\Omega \operatorname{curl}u \cdot \varphi \, dx - \int_\Omega u \cdot \operatorname{curl}\varphi \, dx = \int_{\partial \Omega} (\nu \times u) \cdot \varphi \, d\sigma 
%\end{equation*}
%is valid even for $u$ just in $H(\operatorname{curl},\Omega)$. Note that the right-hand side of the above formula is a reader-friendly notation that in reality stands for the duality pairing between $\nu \times u\rvert_{\partial\Omega} \in H^{-1/2}(\partial \Omega)^3$ and $\varphi\rvert_{\partial\Omega} \in H^{1/2}(\partial \Omega)^3$.
We will also often omit the boundary restriction subscript. 
It turns out that $H_0(\operatorname{curl},\Omega)$ is exactly the space of those vector fields whose tangential trace vanish on $\partial\Omega$ (cf. \cite[Thm. 2.12]{GiRa86}), i.e.
\begin{equation*}
H_0(\operatorname{curl},\Omega) = \set{u \in H(\operatorname{curl},\Omega) : \nu \times u\rvert_{\partial \Omega}=0},
\end{equation*}
hence it naturally encodes the electric boundary condition \eqref{introluz:elec:bc}.
For more details we refer to  \cite[Ch. 2]{GiRa86} or \cite[Ch. IX-A \S 1.2]{DaLi90}.

Similarly, we introduce the space $H(\operatorname{div} \varepsilon, \Omega)$ of vector fields $u\in L^2(\Omega)^3$ such that the vector field $\varepsilon u$ has distributional divergence in $L^2(\Omega)$,
namely there exists a function $\operatorname{div}(\varepsilon u) \in L^2(\Omega)$ such that  
\begin{equation*}  
\int_\Omega \varepsilon u \cdot \nabla \varphi \, dx = - \int_\Omega \operatorname{div}(\varepsilon u) \, \varphi \, dx \qquad \forall\varphi \in C^\infty_c(\Omega).
\end{equation*}
We endow $H(\operatorname{div}\varepsilon, \Omega)$ with the  inner product
\begin{equation*}
\langle u,v\rangle_{H(\operatorname{div} \varepsilon, \Omega)} := \int_\Omega \varepsilon u \cdot v \, dx + \int_\Omega \operatorname{div}(\varepsilon u)  \operatorname{div}(\varepsilon v) \, dx \qquad  \forall u,v \in H(\operatorname{div}\varepsilon,\Omega),
\end{equation*}
which makes it a Hilbert space. 
%that gives norm \[
%\|u\|_{H(\mathrm{div}, \Omega)} = \left( \| u \|^2_{L^2(\Omega)^3} + \| \mathrm{div}\,u\|^2_{L^2(\Omega)} \right)^{1/2} \qquad \forall u \in H(\mathrm{div}, \Omega).
%\]
Moreover,  we consider the space
\[
\xn^\varepsilon(\Omega) := H_0(\mathrm{curl}, \Omega) \cap H(\mathrm{div}\,\varepsilon, \Omega)
\]
equipped with inner product
\begin{equation*}
\langle u,v\rangle_{\xn^\varepsilon(\Omega)} := \int_\Omega \varepsilon u \cdot v \, dx  + \int_\Omega \operatorname{curl}u \cdot \operatorname{curl}v
 \, dx + \int_\Omega \operatorname{div} (\varepsilon u) \, \operatorname{div}(\varepsilon v) \, dx
\end{equation*}
for all $u,v \in \xn^\varepsilon(\Omega)$.
%the norm defined by
%\[
%\|u\|_{\xn(\Omega)} = \left( \| u \|^2_{L^2(\Omega)^3} + \| \mathrm{curl}\,u\|^2_{L^2(\Omega)^3} +\| \mathrm{div}\,u\|^2_{L^2(\Omega)} \right)^{1/2} \qquad \forall u \in \xn(\Omega).
%\]
Finally, we set 
\begin{equation*}
\begin{split}
\xn^\varepsilon(\mathrm{div}\,\varepsilon 0,\Omega) &:= \set{u \in \xn^\varepsilon(\Omega) : \mathrm{div} \, (\varepsilon u) = 0} \\
&= \set{u \in L^2(\Omega)^3 : \operatorname{curl}u \in L^2(\Omega)^3, \operatorname{div}(\varepsilon u)=0, \nu \times u\rvert_{\partial \Omega}=0}.
\end{split}
\end{equation*}
If $\varepsilon \in \mathcal{E}$ and the assumption \eqref{Omega_def} holds, i.e. if $\Omega$ is a bounded domain of $\mathbb{R}^3$ of class $C^{1,1}$, the space $\xn^\varepsilon(\Omega)$ is continuously embedded into $H^1(\Omega)^3$. This is implied by the so-called Gaffney (or Gaffney-Friedrichs) inequality,  which states that there exists a constant $C_\varepsilon >0$ such that
\begin{equation}\label{ineq:GF}
\norm{u}_{H^1(\Omega)^3}^2 \leq C_\varepsilon \left(  \langle \varepsilon u, u\rangle_{L^2(\Omega)^3}+ \norm{\operatorname{curl}u}^2_{L^2(\Omega)^3} + \|\operatorname{div}\varepsilon u\|^2_{L^2(\Omega)} \right) = C_\varepsilon \|u\|_{\xn^\varepsilon (\Omega)}
\end{equation}
for all $u \in \xn^\varepsilon (\Omega)$. 
We refer to Prokhorov and Filonov \cite[Thm. 1.1]{PrFi15} for a proof of the above inequality. Their result includes more general  permittivities and  domains, such as convex domains or in general Lipschitz domains satisfying the exterior ball condition. Another  proof can be found in Alberti and Capdeboscq \cite{AlCa14}. Other classical references for the Gaffney inequality are   Saranen \cite{Sa83} and Mitrea \cite{Mit01}. More recently, Creo and Lancia \cite{CrLa20} generalized the Gaffney inequality to more irregular domains in dimension $2$ and $3$.
Incidentally, we point out that one of the main reasons for the regularity assumption   \eqref{Omega_def} we require on $\Omega$   is exactly the validity of \eqref{ineq:GF}.
Note that if we lower the regularity assumptions, for example requiring $\Omega$ just of Lipschitz class, there is no guarantee to obtain the same perturbation results. The authors plan to address this issue in future works.

We recall here a known formula for the divergence of the matrix-vector product $\varepsilon v$ with $\varepsilon \in \mathcal{E}$ and $v \in H^1(\Omega)^3$ that we will exploit extensively throughout the paper:
\begin{equation} \label{div:matrixxvector}
\operatorname{div}(\varepsilon v) = \operatorname{tr}(\varepsilon Dv) + \operatorname{div}\varepsilon \cdot v 
\qquad \mbox{ a.e. in } \Omega.
\end{equation}
where $\operatorname{tr}(\cdot)$ denotes the trace operator and  $\operatorname{div}\varepsilon$ is the vector field defined by
\begin{equation*}
\operatorname{div}\varepsilon := \left( \operatorname{div}\varepsilon^{(1)}, \operatorname{div}\varepsilon^{(2)}, \operatorname{div}\varepsilon^{(3)} \right).
\end{equation*}
  with $\varepsilon^{(k)}$ denoting the $k$-th column of  $\varepsilon = \left(\varepsilon^{(1)} | \, \varepsilon^{(2)} | \, \varepsilon^{(3)} \right)$.

Recall the electric eigenvalue problem
\begin{equation}\label{prob:eigen1}
\begin{cases}
\operatorname{curl}\operatorname{curl} u = \lambda \, \varepsilon \, u \qquad &\mbox{ in } \Omega,\\
\operatorname{div}\varepsilon u = 0 \qquad &\mbox{ in } \Omega,\\
\nu \times u = 0 \qquad &\mbox{ on } \partial \Omega.
\end{cases}
\end{equation}
By classical integration by parts, one has that
\begin{equation*}
\int_\Omega \operatorname{curl}F \cdot G \, dx = \int_\Omega F \cdot \operatorname{curl}G \, dx + \int_{\partial \Omega} (G \times \nu) \cdot F \, d\sigma
\end{equation*}
for all sufficiently regular vector fields $F,G$ (see, e.g. \cite[Thm. A.13]{kihe}). Then is readily seen  that the weak formulation of problem \eqref{prob:eigen1} is
\begin{equation}\label{prob:eigen1weak}
\int_{\Omega} \mathrm{curl}\,u \cdot \mathrm{curl}\, v \,dx  =\lambda\int_{\Omega}\varepsilon u\cdot v \,dx
  \qquad \forall v \in \xn^\varepsilon(\mathrm{div}\,\varepsilon 0,\Omega),
\end{equation}
in the unknowns $\lambda \in \mathbb{R}$ (the eigenvalues) and $u \in \xn^\varepsilon(\mathrm{div}\,\varepsilon 0,\Omega)$ (the eigenvectors).
The eigenvalues of problem \eqref{prob:eigen1weak} are non-negative, as one can easily see by testing the eigenfunction $u$ against itself.

For our purposes it will be convenient  to work in the space $\xn^\varepsilon(\Omega)$ rather than $\xn^\varepsilon(\mathrm{div}\,\varepsilon 0,\Omega)$. Hence, following Costabel \cite{Co91} and Costabel and Dauge \cite{CoDo99}, we consider the following eigenvalue problem which presents  an additional penalty term: 
\begin{equation}\label{prob:eigen2weak}
\int_{\Omega} \mathrm{curl}\,u \cdot \mathrm{curl}\,v \,dx + \tau \int_{\Omega} \operatorname{div}(\varepsilon u )\, \operatorname{div}(\varepsilon v) \,dx = \sigma \int_{\Omega}\varepsilon u\cdot v \,dx \quad \forall v \in \xn^\varepsilon(\Omega),
\end{equation}
in the unknowns $u \in  \xn^\varepsilon(\Omega)$ and $\sigma \in \mathbb{R}$.
Here $\tau>0$ is any fixed positive real number.
Solutions of problem \eqref{prob:eigen1weak} will then corresponds to solutions $u$ of \eqref{prob:eigen2weak} with $\operatorname{div}(\varepsilon u)=0$ in $\Omega$ (see Theorem \ref{thm:codau} below).
Observe that also the eigenvalues $\sigma$ of problem \eqref{prob:eigen2weak} are non-negative, and that the zero eigenspace of problem \eqref{prob:eigen2weak} (and of problem \eqref{prob:eigen1weak}) coincides with the set $K^\varepsilon(\Omega)$ defined in \eqref{zero:eigenspace:eps}.
%As said in the introduction, we will only consider positive eigenvalues and forget about the zero eigenvalue
%(cf. \Cref{defi:maxeig}).
%
%

Following a standard procedure, one can convert  problem \eqref{prob:eigen2weak} into an eigenvalue problem
for a compact self-adjoint operator. 
Recall the map $J_\varepsilon$ defined in \eqref{eps:inner:product}, which is nothing but the bilinear form corresponding to the inner product of $L^2_\varepsilon(\Omega)$. Obviously $J_\varepsilon$ can be thought as an operator acting from $L^2_\varepsilon(\Omega)$ to $(\xn^\varepsilon (\Omega))'$.
%We define the operator $B_\varepsilon: \xn^\varepsilon(\Omega) \to (\xn^\varepsilon(\Omega))'$ defined by  the left hand side of equation \eqref{prob:eigen2weak}, that is 
%\[
%B_\varepsilon[u][v] := \int_{\Omega} \mathrm{curl}\,u \cdot \mathrm{curl}\,v \,dx + \tau \int_{\Omega} \operatorname{div}(\varepsilon u )\, \operatorname{div}(\varepsilon v) \,dx \quad \forall u,v \in \xn^\varepsilon(\Omega).
%\]
We define the operator $T_\varepsilon$ from $\xn^\varepsilon (\Omega)$ to its dual $(\xn^\varepsilon (\Omega))'$ by
\begin{equation*}
T_\varepsilon[u][v]:= \int_\Omega \varepsilon u \cdot v \, dx + \int_\Omega \operatorname{curl}u \cdot \operatorname{curl}v \, dx + \tau \int_\Omega \operatorname{div}(\varepsilon u) \, \operatorname{div}(\varepsilon v) \, dx \quad \forall u,v \in \xn^\varepsilon(\Omega).
\end{equation*}
Observe that by the Riesz theorem, $T_\varepsilon$ is a homeomorphism from $\xn^\varepsilon (\Omega)$ to its dual and thus  it can be inverted.
We can therefore define the operator  $S_\varepsilon$, acting  from $L^2_\varepsilon(\Omega)$ to itself, by setting 
\begin{equation} \label{operator:dfn}
S_\varepsilon := \iota_\varepsilon \circ T_\varepsilon^{-1} \circ J_\varepsilon: L^2_\varepsilon(\Omega) \to L^2_\varepsilon(\Omega),
\end{equation}
where $\iota_\varepsilon$ denotes the embedding of $\xn^\varepsilon(\Omega)$ into $L^2_\varepsilon (\Omega)$.
Observe that the space $L^2_\varepsilon(\Omega)$ is equal to $L^2(\Omega)^3$ as a set, and the varying inner products depending on $\varepsilon$ are all equivalent to the standard one. We then have the following lemma.
\begin{lemma}\label{lem:Te}
Let $\varepsilon \in \mathcal{E}$.
Then the operator $S_\varepsilon$ is a self-adjoint operator from $L^2_\varepsilon(\Omega)$ to itself. Moreover, $\sigma$ is an eigenvalue of problem \eqref{prob:eigen2weak} if and only if $\mu=(\sigma +1)^{-1}$ is an eigenvalue of the operator $S_\varepsilon$, the eigenvectors being the same.
\end{lemma}
\begin{proof}
Since $J_\varepsilon$ and $T_\varepsilon$ are both symmetric we get that 
\begin{equation*}
\begin{split}
J_\varepsilon[S_\varepsilon [u]][v] &= J_\varepsilon [v] [T_\varepsilon^{-1} \circ J_\varepsilon [u]] = T_\varepsilon [T_\varepsilon^{-1} \circ J_\varepsilon [v]] [T_\varepsilon^{-1} \circ J_\varepsilon [u]] \\
&=  T_\varepsilon[T_\varepsilon^{-1} \circ J_\varepsilon [u]] [T_\varepsilon^{-1} \circ J_\varepsilon[v]] = J_\varepsilon [u] [S_\varepsilon [v]] \quad \forall u,v \in L^2_\varepsilon(\Omega),
\end{split}
\end{equation*}
proving that $S_\varepsilon$ is self-adjoint in $L^2_\varepsilon(\Omega)$.

Finally, if $(\sigma, u) \in \mathbb{R} \times \xn^\varepsilon(\Omega)$ is an eigenpair of problem \eqref{prob:eigen2weak}, then $T_\varepsilon [u]= (\sigma +1) J_\varepsilon[u]$. Viceversa, if $(\mu,u) \in \mathbb{R} \times L^2_\varepsilon(\Omega)$ is such that $S_\varepsilon [u] = \mu u$ then $u \in \xn^\varepsilon(\Omega)$ and $T_\varepsilon [u] = \mu^{-1} J_\varepsilon [u]$, and thus $u$ is an eigenvector of problem \eqref{prob:eigen2weak} corresponding to the eigenvalue $\sigma=\mu^{-1}-1$.
\end{proof}
If the space $\xn^\varepsilon(\Omega)$ is compactly embedded into $L^2(\Omega)^3$, which is true under our assumptions on $\varepsilon$ and $\Omega$ (see Weber \cite{We80}), the operator $S_\varepsilon$ is compact and its spectrum consists of $\{0\} \cup \{\mu_n\}_{n \in \mathbb{N}}$ with $\mu_n$ being a decreasing sequence composed of positive eigenvalues of $S_\varepsilon$ of finite multiplicity converging to zero.
Accordingly, by \Cref{lem:Te}, the spectrum of problem \eqref{prob:eigen2weak} is composed by ($\varepsilon$-dependent) non-negative eigenvalues of finite multiplicity which can be arranged in an increasing sequence
\[
0\leq \sigma_1[\varepsilon] \leq \sigma_2[\varepsilon] \leq \cdots \leq \sigma_n[\varepsilon] \leq \cdots \nearrow +\infty.
\]
Here each eigenvalue is repeated in accordance with its multiplicity. Note that the zero eigenvalue has fixed multiplicity depending only on the topology of $\Omega$. 
By the min-max formula every eigenvalue can be variationally characterized as follows:
\begin{equation} \label{minmax:formula}
\sigma_j[\varepsilon] = \min_{\substack{V_j \subset \xn^\varepsilon(\Omega),\\ \operatorname{dim}V_j = j}} \max_{\substack{u \in V_j,\\u \neq 0}} \frac{\int_{\Omega} \abs{\operatorname{curl}u}^2 dx + \tau \int_\Omega \abs{\operatorname{div}(\varepsilon u)}^2 dx}{\int_{\Omega} \varepsilon u \cdot u \, dx}.
\end{equation}
Moreover, we have the following result, in the same spirit of  Costabel and Dauge \cite[Thm 1.1]{CoDo99}. 
\begin{theorem} \label{thm:codau}
Let $\Omega$ be as in \eqref{Omega_def}. Let $\varepsilon \in \mathcal{E}$. 
Then the eigenpairs $(\sigma, u) \in \mathbb{R} \times \xn^\varepsilon(\Omega)$ of problem 
\eqref{prob:eigen2weak} are spanned by the following two disjoint families:
\begin{itemize}
\item[i)] the pairs $( \lambda,  u)   \in \mathbb{R} \times \xn^\varepsilon(\mathrm{div}\,\varepsilon 0,\Omega)$ solutions of  problem   \eqref{prob:eigen1weak};
\item[ii)] the pairs $(\tau \rho,\nabla f)$ where $(\rho,f) \in \mathbb{R} \times  H^1_0(\Omega)$ is an eigenpair of the problem
\begin{equation} \label{dirichlet:problem:simillap}
\begin{cases}
-\operatorname{div}(\varepsilon \nabla f)=\rho f & \text{in }\Omega,\\
f=0 & \text{on }\partial \Omega.
\end{cases}
\end{equation}
\end{itemize}
In particular, the set of eigenvalues of problem  \eqref{prob:eigen2weak} are given by the union of the set of eigenvalues of problem \eqref{prob:eigen1weak} and the set of eigenvalues of the operator 
$-\operatorname{div}(\varepsilon\nabla\cdot)$ with Dirichlet boundary conditions  in $\Omega$ multiplied by  $\tau$.
\end{theorem}
 
\begin{proof}
It is easily seen that if $(\lambda, u) \in \mathbb{R} \times \xn^\varepsilon(\mathrm{div}\,\varepsilon 0,\Omega)$ is an eigenpair of problem  \eqref{prob:eigen1weak}, then it is an eigenpair of problem \eqref{prob:eigen2weak}. Moreover, if $u = \nabla f$, where $f \in H^1_0(\Omega)$ is a solution of problem \eqref{dirichlet:problem:simillap}, then $u \in \xn^\varepsilon(\Omega)$  solves  \eqref{prob:eigen2weak} with $\sigma=\tau\rho$.

Conversely, suppose that $(\sigma, u) \in \mathbb{R} \times \xn^\varepsilon(\Omega)$  is an eigenpair
 of problem 
 \eqref{prob:eigen2weak}. If 
 \[
p:= \mathrm{div}(\varepsilon u) =0,
\] 
then clearly  $u \in  \xn^\varepsilon(\mathrm{div}\,\varepsilon 0,\Omega)$ and solves \eqref{prob:eigen1weak}. Suppose now that $p \neq 0$.
  We set
\[
H^1_0(\Omega,\mathrm{div}(\varepsilon \nabla\cdot)) := \{u \in H^1_0(\Omega): \mathrm{div}(\varepsilon \nabla u)  \in L^2(\Omega)\}.
\]
Then for all $\psi \in H^1_0(\Omega,\mathrm{div}(\varepsilon \nabla\cdot))$, by taking $\nabla \psi$ 
as test functions in \eqref{prob:eigen2weak} we get
\begin{equation*}
\int_\Omega \tau \, p \, \operatorname{div}(\varepsilon \nabla \psi) \, dx= \sigma \int_\Omega \varepsilon u \cdot \nabla \psi \, dx =  - \sigma \int_\Omega p \, \psi \, dx,
\end{equation*}
thus
\begin{equation} \label{eq:p:and:div}
\int_\Omega p \left( \tau \, \operatorname{div}(\varepsilon \nabla \psi) + \sigma \psi \right) \, dx =0.
\end{equation}
%does the fredholm alternative hold for this operator?
Necessarily $\sigma/\tau$ belongs to the spectrum of the operator $-\operatorname{div}(\varepsilon \nabla\cdot)$ with Dirichlet boundary conditions, because if not we could find a $\hat{\psi}$ such that $\operatorname{div}(\varepsilon \nabla \hat{\psi}) + \frac{\sigma}{\tau}\hat{\psi}=p$, hence from \eqref{eq:p:and:div} we would get $p=0$, which is a contradiction. From the Fredholm alternative we deduce that $p$ belongs to the associated eigenspace, thus $p \in H^1_0(\Omega,\mathrm{div}(\varepsilon \nabla\cdot))$ and 
\begin{equation} \label{diff:equation:for:p}
\operatorname{div}(\varepsilon \nabla p) + \frac{\sigma}{\tau} \, p =0.
\end{equation}
Now, we define the field 
\[
w := u + \frac{\tau}{\sigma} \, \nabla p \in \xn^\varepsilon(\Omega).
\]
If $w=0$ then $u  = - \frac{\tau}{\sigma} \, \nabla p$, and recalling \eqref{diff:equation:for:p} one deduces that 
$(\sigma,u)$ is of the form in ii).
Therefore, suppose that $w \neq 0$. Observe that $w$ satisfies 
\begin{equation*}
\operatorname{div}(\varepsilon w)=p + \frac{\tau}{\sigma} \operatorname{div}(\varepsilon \nabla p) =0 \quad \text{and} \quad \operatorname{curl}w = \operatorname{curl}u.
\end{equation*}
Hence for any $v \in \xn^\varepsilon(\Omega)$ 
\begin{equation*}
\begin{split}
\int_\Omega \operatorname{curl}w \cdot \operatorname{curl}v \, dx &= \int_\Omega \left( \sigma  \, \varepsilon u \cdot v - \tau \, p \, \operatorname{div}(\varepsilon v) \right) dx = \int_\Omega ( \sigma \, \varepsilon u + \tau \, \varepsilon \nabla p) \cdot v \, dx \\
&= \sigma \int_\Omega \varepsilon w \cdot v \, dx.
\end{split}
\end{equation*}
Thus the pair $(\sigma, w)$ belongs to the  family in i) and $\sigma$ is a multiple eigenvalue of \eqref{prob:eigen2weak}. In this case we can split the eigenspace corresponding to $\sigma$ according to the two families in i) and ii).
\end{proof}
In view of the previous theorem, we introduce the following definition.
\begin{definition}\label{defi:maxeig}
Let $\Omega$ be as in \eqref{Omega_def}. Let $\varepsilon \in \mathcal{E}$.
An eigenvalue $\sigma$ of problem \eqref{prob:eigen2weak} is said to be 
  a \emph{Maxwell eigenvalue} if  there exists $u \in \xn^\varepsilon(\mathrm{div}\,\varepsilon 0,\Omega)$, $u \neq 0$, such that $(\sigma, u)  $ is an eigenpair of problem \eqref{prob:eigen1weak}.
  In this case, we say that $u$ is a \emph{Maxwell eigenvector}. We denote the set of Maxwell eigenvalues by:
 \[
0\leq \lambda_1[\varepsilon] \leq \lambda_2[\varepsilon] \leq \cdots \leq \lambda_n[\varepsilon] \leq \cdots \nearrow +\infty,
\]
where we repeat the eigenvalues in accordance with their (Maxwell) multiplicity,  i.e. the dimension of the space generated by the corresponding Maxwell eigenvectors. 
\end{definition}

We stress that the introduction of problem  \eqref{prob:eigen2weak} is of technical nature to bypass the problem of working in $\varepsilon$-dependent spaces, but in this  paper we are mostly interested in the behavior of Maxwell eigenvalues. Accordingly, we will focus more on the behavior of  $\{\lambda_j[\varepsilon]\}_{j \in \mathbb{N}} \subseteq  \{\sigma[\varepsilon]\}_{j \in \mathbb{N}}$ than on the 
one of all  $\{\sigma[\varepsilon]\}_{j \in \mathbb{N}}$.
 Note also that the Maxwell eigenvalues $\{\lambda_j[\varepsilon]\}_{j \in \mathbb{N}}$ do not depend upon the choice of the parameter $\tau>0$ multiplying  the penalty 
term of problem  \eqref{prob:eigen2weak}, meaning that different values of $\tau$ provide exactly the same Maxwell spectrum.

\section{Continuity of the eigenvalues}\label{sec:cont}

We first focus on the continuity of the eigenvalues $\sigma_j[\varepsilon]$ of problem \eqref{prob:eigen2weak}, which in particular
implies the continuity of the Maxwell eigenvalues $\lambda_j[\varepsilon]$. 
For the sake of simplicity, in this section we will fix $\tau=1$. Note that the results presented below remain valid independently of the value of $\tau >0$.

We find it convenient to introduce the space
\begin{equation} \label{dfn:H1N}
H^1_{\rm \scriptscriptstyle N}(\Omega) : = \set{u \in H^1(\Omega)^3 : \nu \times u =0 \text{ on } \partial \Omega},
\end{equation}
endowed with the usual $H^1$-norm. 
Note that in view of  formula \eqref{div:matrixxvector}  and of the Gaffney inequality \eqref{ineq:GF} (valid under our assumptions \eqref{Omega_def})   the spaces $\xn^\varepsilon(\Omega)$ and $H^1_{\rm \scriptscriptstyle N}(\Omega)$ coincide as sets  for every $\varepsilon \in \mathcal{E}$, and their respective norms are equivalent.
Hence one can use the space $H^1_{\rm \scriptscriptstyle N}(\Omega)$ for the variational characterization of the eigenvalues: the benefit lies in the fact that in this way we do not have to deal with Hilbert spaces that may depend on the permittivity parameter $\varepsilon$, allowing us to compare Rayleigh quotients  relative to different permittivities.
In other words, the min-max characterization \eqref{minmax:formula} can be equivalently written as
\begin{equation} \label{minmax:formula:H1}
\sigma_j[\varepsilon] = \min_{\substack{V_j \subset H^1_{\rm \scriptscriptstyle N}(\Omega),\\ \operatorname{dim}V_j = j}} \max_{\substack{u \in V_j,\\u \neq 0}} \frac{\int_{\Omega} \abs{\operatorname{curl}u}^2 dx +  \int_\Omega \abs{\operatorname{div}(\varepsilon u)}^2 dx}{\int_{\Omega} \varepsilon u \cdot u \, dx},
\end{equation}
which is the one we will exploit in order to prove our continuity result.

Before doing so, we first prove a locally uniform Gaffney inequality, that can be obtained by exploiting the standard inequality \eqref{ineq:GF} for a fixed permittivity.
\begin{proposition}\label{prop:ugi}
Let $\Omega$ be as in \eqref{Omega_def}. Let $\tilde{\varepsilon} \in \mathcal{E}$. Then there exist two constants $\delta,  C_\mathcal{G}>0$ such that  \begin{equation} \label{uniform:permittivity:gaffney:ineq}
\|u\|_{H^1(\Omega)^3}^2 \leq C_\mathcal{G} \left( \langle \varepsilon u, u\rangle_{L^2(\Omega)^3} + \|\operatorname{curl}u\|^2_{L^2(\Omega)^3} + \|\operatorname{div}(\varepsilon u)\|^2_{L^2(\Omega)} \right)
\end{equation}
for all $u \in H^1_{\rm \scriptscriptstyle N}(\Omega)$ and for all $\varepsilon \in \mathcal{E}$ with $\|\varepsilon - \tilde{\varepsilon}\|_{W^{1,\infty}(\Omega)} < \delta$.
\end{proposition} 
\begin{proof}
First of all, we observe that  if $\varepsilon' \in \mathcal{E}$ then 
by formula \eqref{div:matrixxvector}  we have that 
\begin{equation}\label{PK}
\operatorname{div}(\varepsilon' u) = \operatorname{tr}(\varepsilon' Du) + (\operatorname{div}\varepsilon') \cdot u.
\end{equation}
Moreover, if  $M$ is a $3 \times 3$ matrix then the following inequalities
\begin{equation} \label{stima:trace}
|\operatorname{tr}(\varepsilon'(x) M)| \leq 9 \norm{\varepsilon'}_{L^\infty(\Omega)} |M|,
\end{equation}
\begin{equation} \label{stima:diveps}
|\operatorname{div}\varepsilon'(x)| \leq 3 \sqrt{3} \norm{\varepsilon'}_{W^{1,\infty}(\Omega)}
\end{equation}
hold for a.e. $x \in \Omega$, where $|M|$ denotes the matrix norm $|M| := \max_{i,j}|M_{ij}|$.

Fix  $u \in H^1_{\rm \scriptscriptstyle N}(\Omega)$ and $\varepsilon \in \mathcal{E}$.
From \eqref{ineq:GF} we have that the Gaffney inequality holds for $\tilde{\varepsilon}$, namely there exists a constant $C_{\tilde{\varepsilon}} >0$ independent of $u$ such that
\begin{equation} \label{gaffney:ineq:tilde:eps}
\|u\|_{H^1(\Omega)^3}^2 \leq C_{\tilde{\varepsilon}} \left( \langle \tilde\varepsilon u, u\rangle_{L^2(\Omega)^3} + \|\operatorname{curl}u\|^2_{L^2(\Omega)^3} + \|\operatorname{div}\tilde\varepsilon u\|^2_{L^2(\Omega)} \right).
\end{equation}
%We now take a $\delta>0$ that will be specified later.
     %such that $\|\varepsilon - \tilde{\varepsilon}\|_{W^{1,\infty}(\Omega)}<\delta$. 
  %From now until the end of the proof, with $C$ we denote a dimensional constant which may vary from line to line.
Moreover
\begin{align*}
\abs{\operatorname{tr}(\tilde\varepsilon Du)^2 - \operatorname{tr}(\varepsilon Du)^2} &= \abs{\operatorname{tr}\left( (\tilde\varepsilon + \varepsilon) Du\right) \operatorname{tr}\left( (\tilde\varepsilon - \varepsilon) Du) \right)} 
\\
&  \leq 9^2 \, \norm{\tilde\varepsilon +\varepsilon}_{L^\infty(\Omega)} \norm{\tilde\varepsilon -\varepsilon}_{L^\infty(\Omega)} \, |Du|^2,
\end{align*}
and
\begin{align*}
\abs{\left( \operatorname{div}\tilde\varepsilon \cdot u\right)^2 - \left( \operatorname{div}\varepsilon \cdot u\right)^2}& = \abs{\left( \operatorname{div}(\tilde\varepsilon- \varepsilon) \cdot u \right) \left( \operatorname{div} (\tilde\varepsilon + \varepsilon) \cdot u  \right)} \\
 &\leq (3\sqrt{3})^2 \, \norm{\tilde\varepsilon +\varepsilon}_{W^{1,\infty}(\Omega)} \norm{\tilde\varepsilon -\varepsilon}_{W^{1,\infty}(\Omega)} \, |u|^2
\end{align*}
and 
\begin{equation*}
\begin{split}
&2 \abs{\operatorname{tr}(\tilde\varepsilon Du) \ \operatorname{div}\tilde\varepsilon \cdot u - \operatorname{tr}(\varepsilon Du) \ \operatorname{div}\varepsilon \cdot u} \\
& \quad \leq 2 \abs{\operatorname{tr}(\tilde\varepsilon Du) \  \operatorname{div}(\tilde\varepsilon- \varepsilon) \cdot u} + 2 \abs{\operatorname{tr}\left((\tilde\varepsilon-\varepsilon) Du \right) \ \operatorname{div}\varepsilon \cdot u}
\\
& \quad \leq 2\cdot 9\cdot3\sqrt{3} \, \left(\norm{\tilde\varepsilon}_{W^{1,\infty}(\Omega)} + \norm{\tilde\varepsilon -\varepsilon}_{W^{1,\infty}(\Omega)}  \right) \norm{\tilde\varepsilon -\varepsilon}_{W^{1,\infty}(\Omega)} 2 \, |u|\,|Du|  \\
& \quad \leq 54\sqrt{3} \, \left(\norm{\tilde\varepsilon}_{W^{1,\infty}(\Omega)} + \norm{\tilde\varepsilon -\varepsilon}_{W^{1,\infty}(\Omega)}  \right) \norm{\tilde\varepsilon -\varepsilon}_{W^{1,\infty}(\Omega)} (\,|u|^2 + |Du|^2).
\end{split}
\end{equation*}
Thus
\begin{equation} \label{diff:div:estimate}
\begin{split}
&\abs{\,\norm{\operatorname{div}(\tilde{\varepsilon}u)}_{L^2(\Omega)}^2 - \norm{\operatorname{div}(\varepsilon u)}_{L^2(\Omega)}^2} \\
& \qquad \leq  54\sqrt{3} \, \left(\norm{\tilde\varepsilon}_{W^{1,\infty}(\Omega)} + \norm{\tilde\varepsilon -\varepsilon}_{W^{1,\infty}(\Omega)} + \norm{\tilde\varepsilon +\varepsilon}_{W^{1,\infty}(\Omega)}  \right)\\
& \qquad\qquad \times  \norm{\tilde\varepsilon -\varepsilon}_{W^{1,\infty}(\Omega)} \left (\int_\Omega |u|^2 + \int_\Omega |Du|^2 \right).
\end{split}
\end{equation}
Moreover, we have that  
\begin{equation} \label{stima:diff:eps:tildeeps:L2}
\abs{\langle \tilde\varepsilon u, u\rangle_{L^2(\Omega)^3} - \langle \varepsilon u, u\rangle_{L^2(\Omega)^3}} \leq 3 \norm{\tilde{\varepsilon} - \varepsilon}_{L^\infty(\Omega)} \norm{u}^2_{L^2(\Omega)^3}.
\end{equation}
Therefore, making use of \eqref{diff:div:estimate} and \eqref{stima:diff:eps:tildeeps:L2} in \eqref{gaffney:ineq:tilde:eps} 
we obtain that
\begin{equation*}
\begin{split}
\norm{u}_{H^1(\Omega)^3}^2 &\leq C_{\tilde\varepsilon} \left( \langle \varepsilon u, u\rangle_{L^2(\Omega)^3} + \norm{\operatorname{curl}u}^2_{L^2(\Omega)^3} + \norm{\operatorname{div}\varepsilon u}^2_{L^2(\Omega)} \right) + \\
& + C_{\tilde\varepsilon} \left( \,  54\sqrt{3} \, \left(\norm{\tilde\varepsilon}_{W^{1,\infty}(\Omega)} + \norm{\tilde\varepsilon -\varepsilon}_{W^{1,\infty}(\Omega)} + \norm{\tilde\varepsilon +\varepsilon}_{W^{1,\infty}(\Omega)}  \right) +3\right)\norm{\tilde\varepsilon -\varepsilon}_{W^{1,\infty}(\Omega)} \norm{u}^2_{H^1(\Omega)^3}\\
 &\leq C_{\tilde\varepsilon} \left( \langle \varepsilon u, u\rangle_{L^2(\Omega)^3} + \norm{\operatorname{curl}u}^2_{L^2(\Omega)^3} + \norm{\operatorname{div}\varepsilon u}^2_{L^2(\Omega)} \right) + \\
& + C_{\tilde\varepsilon} \left( \,  3\cdot 54\sqrt{3} \, \left(\norm{\tilde\varepsilon}_{W^{1,\infty}(\Omega)} + \norm{\tilde\varepsilon -\varepsilon}_{W^{1,\infty}(\Omega)}    \right)+3\right) \norm{\tilde\varepsilon -\varepsilon}_{W^{1,\infty}(\Omega)} \norm{u}^2_{H^1(\Omega)^3}.
\end{split}
\end{equation*}
Hence, taking $\delta>0$ small enough such that for all $\varepsilon \in \mathcal{E}$ with $\|\tilde\varepsilon -\varepsilon\|_{W^{1,\infty}(\Omega)} <\delta$ we have that
\begin{equation*}
 1-C_{\tilde\varepsilon}  \left( \, 162\sqrt{3} \, \left(\norm{\tilde\varepsilon}_{W^{1,\infty}(\Omega)} + \norm{\tilde\varepsilon -\varepsilon}_{W^{1,\infty}(\Omega)}  \right)+3\right)\norm{\tilde\varepsilon -\varepsilon}_{W^{1,\infty}(\Omega)}  >0,
\end{equation*}
then we get that
formula \eqref{uniform:permittivity:gaffney:ineq} holds with 
\begin{equation*}
C_\mathcal{G} := \frac{ C_{\tilde\varepsilon} }{ 1-\delta \,C_{\tilde\varepsilon}\left(  \, 162\sqrt{3} \, \left(\norm{\tilde\varepsilon}_{W^{1,\infty}(\Omega)} + \delta \right)+3\right)}.
\end{equation*}
\end{proof}

We are now ready to show that the eigenvalues  $\sigma_j[\varepsilon]$ of problem \eqref{prob:eigen2weak} are 
locally Lipschitz continuous in $\varepsilon$.
\begin{theorem} \label{thm:loclip}
Let $\Omega$ be as in \eqref{Omega_def}. Let $j\in\mathbb{N}$ and $\varepsilon_1 \in \mathcal{E}$.  Then there exist two constants  $\delta, \tilde{C}>0$ such that
\begin{equation} \label{eigenarestronglycont}
\abs{\sigma_j[\varepsilon_1] - \sigma_j[\varepsilon_2]} \leq \tilde{C} \norm{\varepsilon_1 -\varepsilon_2}_{W^{1,\infty}(\Omega)}
\end{equation}
for all $\varepsilon_2 \in \mathcal{E}$ such that $\norm{\varepsilon_1 -\varepsilon_2}_{W^{1,\infty}(\Omega)} < \delta$.
\end{theorem}

\begin{proof}
%By $C$ we denote a dimensional constant which may vary from line to line.
For the sake of simplicity in this proof, given $\varepsilon \in \mathcal{E}$ and  $u \in H^1_{\rm \scriptscriptstyle N}(\Omega)$,
 we set 
 \[
 \mathcal{R}[u]:= \int_\Omega |\operatorname{curl}u|^2 dx , \qquad \mathcal{D}_\varepsilon[u]:= \int_\Omega |\operatorname{div}(\varepsilon u)|^2 dx .
 \]
Let be $\delta>0$ be as in Proposition \ref{prop:ugi} with $\tilde \varepsilon = \varepsilon_1$.
Let $\varepsilon_2 \in \mathcal{E}$ be such that   $\|\varepsilon_1-\varepsilon_2\|_{W^{1,\infty}(\Omega)}<\delta$ and recall that 
$c_{\varepsilon_1}$, $c_{\varepsilon_2}$ denote
the constants associated with the coercivity of $\varepsilon_1, \varepsilon_2$ respectively (see \eqref{def:cc}).
Fix $u \in H^1_{\rm \scriptscriptstyle N}(\Omega)$. Then
\begin{align} \label{eq:rd1rd2}
&\abs{\frac{\mathcal{R}[u] + \mathcal{D}_{\varepsilon_1}[u]}{\int_\Omega \varepsilon_1 u \cdot u\,dx} - \frac{\mathcal{R}[u] + \mathcal{D}_{\varepsilon_2}[u]}{\int_\Omega \varepsilon_2 u \cdot u\,dx}}
\\  \nonumber
& \quad \leq \frac{\mathcal{R}[u] \abs{\int_\Omega (\varepsilon_2- \varepsilon_1) u \cdot u\,dx} + \abs{\mathcal{D}_{\varepsilon_1}[u] \int_\Omega \varepsilon_2 \, u \cdot u\,dx - \mathcal{D}_{\varepsilon_2}[u] \int_\Omega \varepsilon_1 \, u \cdot u\,dx}}{(\int_\Omega \varepsilon_1 u \cdot u\,dx) (\int_\Omega \varepsilon_2 u \cdot u\,dx)}
\\  \nonumber
& \quad \leq \frac{3 \, \|\varepsilon_2- \varepsilon_1\|_{L^\infty(\Omega)} \mathcal{R}[u] \int_\Omega \abs{u}^2\,dx}{(\int_\Omega \varepsilon_1 u \cdot u\,dx) (\int_\Omega \varepsilon_2 u \cdot u\,dx)}
\\  \nonumber
& \quad \quad  + \frac{\Big|\mathcal{D}_{\varepsilon_1}[u] \int_\Omega \varepsilon_2 \, u \cdot u\,dx - \mathcal{D}_{\varepsilon_1}[u] \int_\Omega \varepsilon_1 \, u \cdot u\,dx \Big|}{(\int_\Omega \varepsilon_1 u \cdot u\,dx) (\int_\Omega \varepsilon_2 u \cdot u\,dx)}\\ \nonumber
&\quad\quad+\frac{\Big|\mathcal{D}_{\varepsilon_1}[u] \int_\Omega \varepsilon_1 \, u \cdot u\,dx - \mathcal{D}_{\varepsilon_2}[u] \int_\Omega \varepsilon_1 \, u \cdot u\,dx\Big|}{(\int_\Omega \varepsilon_1 u \cdot u\,dx) (\int_\Omega \varepsilon_2 u \cdot u\,dx)}
\\ \nonumber
& \quad \leq \frac{3 \, \|\varepsilon_1- \varepsilon_2\|_{W^{1,\infty}(\Omega)}}{c_{\varepsilon_2}} \frac{\mathcal{R}[u] + \mathcal{D}_{\varepsilon_1}[u]}{\int_\Omega \varepsilon_1 u \cdot u\,dx}
+ \frac{\abs{\mathcal{D}_{\varepsilon_1}[u] - \mathcal{D}_{\varepsilon_2}[u]}}{\int_\Omega \varepsilon_2 u \cdot u\,dx}.
\end{align}
We now focus on the second term in the right hand side of the above inequality.
By the same reasoning used to prove inequality  \eqref{diff:div:estimate} we deduce that there exist a constant $C>0$ 
not depending on $\varepsilon_1$, $\varepsilon_2$ and $u$ such that
\begin{equation} \label{diff:div:estimate:formula2}
\begin{split}
&\abs{\mathcal{D}_{\varepsilon_1}[u] - \mathcal{D}_{\varepsilon_2}[u]} \leq C \, \max_{i=1,2}\set{\norm{\varepsilon_i}_{W^{1,\infty}(\Omega)}} \norm{\varepsilon_1 -\varepsilon_2}_{W^{1,\infty}(\Omega)} \left (\int_\Omega |u|^2\,dx + \int_\Omega |Du|^2 \,dx\right).
\end{split}
\end{equation}
Moreover, thanks to the locally uniform Gaffney inequality \eqref{uniform:permittivity:gaffney:ineq} there exists a constant $C_\mathcal{G}>0$ such that for $i=1,2$ 
\begin{equation*}
\int_\Omega |Du|^2\,dx \leq C_\mathcal{G}  \int_\Omega \left( \varepsilon_i u \cdot u + |\operatorname{curl}u|^2 + |\operatorname{div}(\varepsilon_i u)|^2 \right)\,dx.
\end{equation*}
Using the above inequality with $i=2$ we get
\begin{equation*}
\frac{\int_\Omega |Du|^2\,dx}{\int_\Omega \varepsilon_2 u \cdot u\,dx} \leq C_\mathcal{G} \left( 1 + \frac{\mathcal{R}[u] + \mathcal{D}_{\varepsilon_2}[u]}{\int_\Omega \varepsilon_2 u \cdot u\,dx} \right),
\end{equation*}
which applied to \eqref{diff:div:estimate:formula2} yields
\begin{align}\label{eq:d1d2}
\frac{\abs{\mathcal{D}_{\varepsilon_1}[u] - \mathcal{D}_{\varepsilon_2}[u]}}{\int_\Omega \varepsilon_2 u \cdot u\,dx} \leq C \, \max_{i=1,2} &\,\set{\|\varepsilon_i\|_{W^{1,\infty}(\Omega)}}  \|\varepsilon_1 -\varepsilon_2\|_{W^{1,\infty}(\Omega)} \\ \nonumber
&\times \left( \frac{1}{c_{\varepsilon_2}} + C_\mathcal{G} \left( 1 + \frac{\mathcal{R}[u] + \mathcal{D}_{\varepsilon_2}[u]}{\int_\Omega \varepsilon_2 u \cdot u\,dx} \right) \right).
\end{align}
Thus it follows from \eqref{eq:rd1rd2} and  \eqref{eq:d1d2}  that
\begin{equation} \label{carabinieri:autov:gigi}
\begin{split}
&\frac{\mathcal{R}[u] + \mathcal{D}_{\varepsilon_1}[u]}{\int_\Omega \varepsilon_1 u \cdot u\,dx} \left( 1 - 3\, \frac{\norm{\varepsilon_2- \varepsilon_1}_{W^{1,\infty}(\Omega)}}{c_{\varepsilon_2}} \right) 
\\
& \quad \leq \frac{\mathcal{R}[u] + \mathcal{D}_{\varepsilon_2}[u]}{\int_\Omega \varepsilon_2 u \cdot u\,dx} \left( 1 + C_\mathcal{G} \, C \max_{i=1,2} \set{\norm{\varepsilon_i}_{W^{1,\infty}(\Omega)}} \norm{\varepsilon_1 -\varepsilon_2}_{W^{1,\infty}(\Omega)} \right)
\\
& \quad \quad + C \max_{i=1,2} \set{\norm{\varepsilon_i}_{W^{1,\infty}(\Omega)}} \norm{\varepsilon_1 -\varepsilon_2}_{W^{1,\infty}(\Omega)} \left( \frac{1}{c_{\varepsilon_2}} + C_\mathcal{G} \right).
\end{split}
\end{equation}
Eventually taking  a smaller $\delta>0$, 
%we have that $\norm{\varepsilon_2- \varepsilon_1}_{W^{1,\infty}(\Omega)}$ is sufficiently small so by that
and taking the appropriate  supremum and infimum in \eqref{carabinieri:autov:gigi}, the min-max formula \eqref{minmax:formula:H1} yields

\begin{equation*}% \label{strong:cont:eigenvalues:noabs}
\begin{split}
\sigma_j[\varepsilon_1] - \sigma_j[\varepsilon_2] &\leq \Bigg( \frac{3}{c_{\varepsilon_2}} \sigma_j[\varepsilon_1]+ C_\mathcal{G} \, C \max_{i=1,2} \set{\norm{\varepsilon_i}_{W^{1,\infty}(\Omega)}} \sigma_j[\varepsilon_2]
\\
& \quad + C \max_{i=1,2} \set{\norm{\varepsilon_i}_{W^{1,\infty}(\Omega)}}\left( \frac{1}{c_{\varepsilon_2}} + C_\mathcal{G} \right) \Bigg) \|\varepsilon_1 -\varepsilon_2\|_{W^{1,\infty}(\Omega)}.
\end{split}
\end{equation*}
Exchanging the role  of $\varepsilon_1$ and $\varepsilon_2$,
we   get the inequality \eqref{eigenarestronglycont} but with a constant possibly depending also on $\varepsilon_2$, which is: 
\begin{align}\label{lipconst}
\widehat C(\varepsilon_2):= & \,\,3\max \set{\frac{\sigma_j[\varepsilon_1]}{c_{\varepsilon_2}}, \frac{\sigma_j[\varepsilon_2]}{c_{\varepsilon_1}}} \\ \nonumber
&+ C\max_{i=1,2} \set{\norm{\varepsilon_i}_{W^{1,\infty}(\Omega)}} \bigg( C_G\max_{i=1,2} \set{\sigma_j[\varepsilon_i]} + \max_{i=1,2} \set{\frac{1}{c_{\varepsilon_i}}} + C_\mathcal{G} \bigg) .
\end{align}
In order to finish the proof, it only remains  to show that this constant can  be chosen uniform in $\varepsilon_2$.  Up to taking a smaller $\delta$, 
we note that by \eqref{contcoerc} the constant $c_{\varepsilon_2}$ is uniformly bounded away from zero in
$\varepsilon_2$. Indeed by  \eqref{contcoerc} 
one has that
\[
c_{\varepsilon_2}\geq c_{\varepsilon_1} -3\delta.
\]
Moreover, $\sigma_j[\varepsilon_2]$ is also locally uniformly bounded in $\varepsilon_2$. Indeed, 
from \eqref{PK}, \eqref{stima:trace} and \eqref{stima:diveps} it is not difficult to see that  there exists a constant $C'>0$ not depending on $\varepsilon_2$ such that
for all $u \in H^1_{\rm \scriptscriptstyle N}(\Omega)$ one has 
\begin{equation*}
\int_\Omega \abs{\operatorname{div}(\varepsilon_2 u)}^2\,dx \leq  C' \norm{\varepsilon_2}_{W^{1,\infty}(\Omega)}^2 \int_\Omega \left( |u|^2 + |Du|^2 \right)\,dx.
\end{equation*}
Then, applying the standard Gaffney inequality (with unitary permittivity) we get  that for all $u \in H^1_{\rm \scriptscriptstyle N}(\Omega)$:
\begin{equation*}
\int_\Omega \abs{\operatorname{div}(\varepsilon_2 u)}^2 \,dx\leq C' \norm{\varepsilon_2}_{W^{1,\infty}(\Omega)}^2 \int_\Omega \left( |u|^2 + |\operatorname{curl}u|^2 + |\operatorname{div}u|^2 \right)\,dx.
\end{equation*}
Hence, using the min-max formula \eqref{minmax:formula:H1} for $\sigma_j[\varepsilon_2]$ we have that
\begin{equation*}
\begin{split}
\sigma_j [\varepsilon_2] &= \min_{\substack{V_j \subset H^1_{\rm \scriptscriptstyle N}(\Omega),\\ \operatorname{dim}V_j = j}} \max_{\substack{u \in V_j,\\u \neq 0}} \frac{\int_\Omega |\operatorname{curl}u|^2\,dx + \int_\Omega |\operatorname{div}(\varepsilon_2 u)|^2\,dx}{\int_\Omega \varepsilon_2 \, u \cdot u\,dx} 
\\
& \leq  \, \frac{C'\norm{\varepsilon_2}^2_{W^{1,\infty}(\Omega)}+1}{c_{\varepsilon_2}} \min_{\substack{V_j \subset H^1_{\rm \scriptscriptstyle N}(\Omega),\\ \operatorname{dim}V_j = j}} \max_{\substack{u \in V_j,\\u \neq 0}} \left( \frac{\int_\Omega |\operatorname{curl}u|^2\,dx + \int_\Omega |\operatorname{div}u|^2\,dx}{\int_\Omega |u|^2\,dx} +1 \right) 
\\
&= \, \frac{C'\norm{\varepsilon_2}^2_{W^{1,\infty}(\Omega)}+1}{c_{\varepsilon_2}} (\sigma_j[I_3] +1)
\\
&\leq \, \frac{C'\left(\norm{\varepsilon_1}_{W^{1,\infty}(\Omega)}+\delta\right)^2+1}{c_{\varepsilon_1}-3\delta} (\sigma_j[I_3] +1),
\end{split}
\end{equation*}
  where  $\sigma_j[I_3]$ is the $j$-th eigenvalue of problem \eqref{prob:eigen2weak} set with unitary permittivity.
 Accordingly, the constant  $\widehat C(\varepsilon_2)$ defined in \eqref{lipconst} is bounded above by a constant 
 $\tilde C$ independent 
 of $\varepsilon_2$ for all $\varepsilon_2 \in \mathcal{E}$ such that $\|\varepsilon_1-\varepsilon_2\|_{W^{1,\infty}(\Omega)}<\delta$. Thus the inequality \eqref{eigenarestronglycont} is proved.
\end{proof}

\section{Analyticity and the derivative in $\varepsilon$}\label{sec:an}

In the previous section we have showed that the eigenvalues $\sigma_j[\varepsilon]$ of the modified problem 
\eqref{prob:eigen2weak} (and in particular the Maxwell eigenvalues $\lambda_j[\varepsilon]$) are locally Lipschitz continuous in  $\varepsilon \in \mathcal{E}$. In this section we are interested in proving higher regularity properties.
More in detail we plan to show  that the eigenvalues depend analytically upon $\varepsilon$, and provide
an explicit formula for their $\varepsilon$-derivative. As already mentioned in the introduction, 
 if we consider  a multiple eigenvalue, a perturbation of the permittivity can in principle   split the eigenvalue into different eigenvalues of lower multiplicity
 and thus the corresponding branches can have a corner at the splitting point. In this case we will not even have differentiability. Our strategy  in order to bypass this problem is to consider the symmetric functions of multiple 
 eigenvalues.  This point of view has been first introduced by Lamberti and Lanza de Cristoforis in  \cite{LaLa04}
  and later successfully  adopted in many other works (see,  e.g.,   \cite{BuLa13, BuLa15, LaZa20, LaLuMu21, LaMuTa21}).

Recall that 
\[
0< \sigma_1[\varepsilon] \leq \sigma_2[\varepsilon] \leq \cdots \leq \sigma_n[\varepsilon] \leq \cdots \nearrow +\infty.
\]
are the eigenvalues of problem  \eqref{prob:eigen2weak}, while instead
\[
0< \lambda_1[\varepsilon] \leq \lambda_2[\varepsilon] \leq \cdots \leq \lambda_n[\varepsilon] \leq \cdots \nearrow +\infty.
\]
are the subset of  Maxwell eigenvalues of problem  \eqref{prob:eigen2weak} (see Definition \ref{defi:maxeig}).
Also recall that, by Lemma \ref{lem:Te}, $\{\sigma_j[\varepsilon]\}_{j \in \mathbb{N}}$ coincide with the reciprocal minus one of the eigenvalues of the operator $S_\varepsilon$ defined in \eqref{operator:dfn}.
In order to obtain an explicit formula for the derivatives  of the Maxwell eigenvalues with respect to the permittivity $\varepsilon$ we need the following technical lemma.

\begin{lemma}\label{lem:der}
Let $\Omega$ be as in \eqref{Omega_def}.
Let $\tilde{\varepsilon} \in \mathcal{E}$ and $\tilde{u}, \tilde{v} \in \xn^{\tilde \varepsilon}(\mathrm{div}\,\tilde\varepsilon 0,\Omega)$ be two Maxwell eigenvectors associated with a Maxwell eigenvalue $\tilde{\lambda}$ with permittivity $\tilde \varepsilon$.   Then
\begin{equation} \label{der:oper}
\langle d\rvert_{\varepsilon = \tilde{\varepsilon}} S_\varepsilon [\eta][\tilde{u}], \tilde{v} \rangle_{\tilde{ \varepsilon}} = \tilde{\lambda}(\tilde{\lambda}+1)^{-2} \int_\Omega \eta \tilde{u} \cdot \tilde{v}\, dx
\end{equation}
for all $\eta \in W^{1,\infty} \left(\Omega\right) \cap  \mathrm{Sym}_3 (\Omega)$.
\end{lemma}
\begin{proof}
Under our assumptions on $\Omega$, the space $\xn^\varepsilon(\Omega)$ coincides with the space $H^1_{\rm \scriptscriptstyle N}(\Omega)$ introduced in \eqref{dfn:H1N}, and their norm are equivalent. Then, it is easily seen that   the compact self-adjoint operator $S_\varepsilon$ in $L^2(\Omega)$ 
is obtained by compositions and inversions of real-analytic maps in $\varepsilon$ (such as linear and multilinear
continuous maps). As a consequence $S_\varepsilon$  depends real analytically upon $\varepsilon$.

Now let  $\eta \in W^{1,\infty} \left(\Omega\right) \cap  \mathrm{Sym}_3 (\Omega)$.
Since $J_{\tilde \varepsilon}[\tilde{u}] = (\tilde{\lambda} +1)^{-1} T_{\tilde{\varepsilon}}[\tilde{u}]$, $J_\varepsilon[\tilde{v}] = (\tilde{\lambda} +1)^{-1} T_{\tilde{\varepsilon}}[\tilde{v}]$, and $S_{\tilde{\varepsilon}}$ is symmetric, we have that 
\begin{equation} \label{diff:JeB}
\begin{split}
&\langle d\rvert_{\varepsilon = \tilde{\varepsilon}} S_\varepsilon [\eta][\tilde{u}], \tilde{v}\rangle_{\tilde \varepsilon} \\
&\quad = \langle \iota_\varepsilon \circ T_{\tilde{\varepsilon}}^{-1} \circ d\rvert_{\varepsilon = \tilde{\varepsilon}} J_\varepsilon [\eta] [\tilde{u}], \tilde{v} \rangle_{\tilde \varepsilon} + \langle \iota_\varepsilon \circ d\rvert_{\varepsilon = \tilde{\varepsilon}} T_{\varepsilon}^{-1} [\eta] \circ J_{\tilde \varepsilon} [\tilde{u}], \tilde{v}\rangle_{\tilde \varepsilon}\\
&\quad = J_{\tilde{\varepsilon}} [\tilde{v}] \left[ \iota_\varepsilon \circ T_{\tilde{\varepsilon}}^{-1} \circ d\rvert_{\varepsilon = \tilde{\varepsilon}} J_\varepsilon[\eta][\tilde{u}] \right] + J_{\tilde{\varepsilon}} [\tilde{v}]\left[ \iota_\varepsilon \circ d\rvert_{\varepsilon = \tilde{\varepsilon}} T_\varepsilon^{-1} [\eta] \circ J_{\tilde{\varepsilon}} [\tilde{u}] \right]\\
&\quad = (\tilde{\lambda} +1)^{-1} T_{\tilde{\varepsilon}}[\tilde{v}] \left[ T_{\tilde{\varepsilon}}^{-1} \circ d\rvert_{\varepsilon = \tilde{\varepsilon}} J_\varepsilon[\eta][\tilde{u}] - T_{\tilde{\varepsilon}}^{-1} \circ d\rvert_{\varepsilon = \tilde{\varepsilon}} T_\varepsilon [\eta] \circ T_{\tilde{\varepsilon}}^{-1} \circ J_{\tilde{\varepsilon}}[\tilde{u}] \right]\\
&\quad = (\tilde{\lambda} +1)^{-1} T_{\tilde{\varepsilon}} \left[ T_{\tilde{\varepsilon}}^{-1} \circ d\rvert_{\varepsilon = \tilde{\varepsilon}} J_\varepsilon[\eta][\tilde{u}] - T_{\tilde{\varepsilon}}^{-1} \circ d\rvert_{\varepsilon = \tilde{\varepsilon}} T_\varepsilon [\eta] \circ T_{\tilde{\varepsilon}}^{-1} \circ (\tilde{\lambda} +1)^{-1} T_{\tilde{\varepsilon}} [\tilde{u}] \right] [\tilde{v}]\\
&\quad = (\tilde{\lambda} +1)^{-1}  \left( d\rvert_{\varepsilon = \tilde{\varepsilon}} J_\varepsilon[\eta][\tilde{u}][\tilde{v}] - (\tilde{\lambda} +1)^{-1} d\rvert_{\varepsilon = \tilde{\varepsilon}} T_\varepsilon [\eta][\tilde{u}][\tilde{v}] \right).
\end{split}
\end{equation}
Moreover, by standard calculus,
\begin{equation} \label{der:J}
d\rvert_{\varepsilon = \tilde{\varepsilon}} J_\varepsilon[\eta][\tilde{u}][\tilde{v}] = \int_\Omega \eta \tilde{u} \cdot \tilde{v} \, dx
\end{equation}
and
\begin{equation} \label{der:B}
\begin{split}
d\rvert_{\varepsilon = \tilde{\varepsilon}} T_\varepsilon [\eta][\tilde{u}][\tilde{v}] = \int_\Omega \eta \tilde{u} \cdot \tilde{v} \, dx + \int_\Omega \left( \operatorname{div}(\tilde{\varepsilon}\tilde{u}) \, \operatorname{div}(\eta \tilde{v}) + \operatorname{div}(\eta \tilde{u}) \, \operatorname{div}(\tilde{\varepsilon} \tilde{v}) \right) \, dx.
\end{split}
\end{equation}
%since
%\begin{equation*}
%d\rvert_{\varepsilon = \tilde{\varepsilon}} \left(\int_\Omega \operatorname{div} ({\color{red}\varepsilon}\tilde{u}) \, \operatorname{div}({\color{red}\varepsilon} \tilde{v}) \, dx\right){\color{red}[\eta]} = \int_\Omega \left( \operatorname{div}(\tilde{\varepsilon}\tilde{u}) \, \operatorname{div}(\eta \tilde{v}) + \operatorname{div}(\eta \tilde{u}) \, \operatorname{div}(\tilde{\varepsilon} \tilde{v}) \right) \, dx.
%\end{equation*}
Since $\operatorname{div}(\tilde{\varepsilon}\tilde{u})=0=\operatorname{div}(\tilde{\varepsilon}\tilde{v})$ in $\Omega$, using \eqref{diff:JeB}, \eqref{der:J} and \eqref{der:B}, we get \eqref{der:oper}.
\end{proof}

Following \cite{LaLa04}, given a finite set of indices $F \subset \mathbb{N}$, we consider those  permittivities $\varepsilon \in \mathcal{E}$ for which Maxwell eigenvalues with indices in $F$ do not coincide with Maxwell eigenvalues with indices outside $F$. We then introduce the following sets:
\[
\mathcal{E}[F] := \set{\varepsilon \in \mathcal{E} : \lambda_j[\varepsilon] \neq \lambda_l[\varepsilon] \ \forall j \in F, l \in \mathbb{N}\setminus F}
\]
and 
\[
\Theta[F] := \set{\varepsilon \in \mathcal{E}[F] : \lambda_j[\varepsilon] \text{ have a common value } \lambda_F[\varepsilon] \text{ for all }  j \in F}.
\]
Let $\varepsilon \in \mathcal{E}[F]$. The elementary symmetric function of degree $s \in \{1,\dots, \abs{F}\}$ of the Maxwell eigenvalues with indices in $F$ is defined by
\[
\Lambda_{F,s}[\varepsilon] := \sum_{\substack{j_1,\dots,j_s \in F \\ j_1<\dots<j_s}} \lambda_{j_1}[\varepsilon] \cdots \lambda_{j_s}[\varepsilon].
\]
In the following theorem we show that the maps $\varepsilon \mapsto \Lambda_{F,s}[\varepsilon]$ are real analytical on $\mathcal{E}[F]$ and we compute their Fr\'echet derivatives with respect to $\varepsilon$.
%Nota personale: qua e' importante che si abbia uno spazio intermedio FISSO H^1_N (da L^2_\varepsilon to H^1_N* to H^1_N to L^2_\varepsilon), e non variabile come X_N^\varepsilon, in modo da poter applicare lanza lamberti, e dedurre la reale-analiticita' 
\begin{theorem}\label{thm:diffeps}
Let $\Omega$ be as in \eqref{Omega_def}. Let $F$ be a finite subset of $\mathbb{N}$ and 
$s \in \{1,\ldots,|F|\}$. Then $\mathcal{E}[F]$ is  open  in $W^{1,\infty} \left(\Omega\right) \cap  \mathrm{Sym}_3 (\Omega)$ and the elementary symmetric function $\Lambda_{F,s}$ depend real analytically upon $\varepsilon \in \mathcal{E}[F]$.

Moreover, if $\{F_1,\ldots, F_n\}$ is a partition 
of $F$ and $\tilde \varepsilon \in \bigcap_{k=1}^n \Theta[F]$ is 
such that for each $k =1,\ldots, n$ the Maxwell eigenvalues $\lambda_j[\tilde \varepsilon]$ assume the common 
value $\lambda_{F_k}[\tilde \varepsilon]$ for all $j \in F_k$, then the differential of the function 
$\Lambda_{F,s}$ at the point $\tilde \varepsilon$ is given by the formula
\begin{equation}\label{diff:Lambdas}
d\rvert_{\varepsilon=\tilde{\varepsilon}} \Lambda_{F,s} [\eta] = -\sum_{k=1}^nc_k  \sum_{l \in F_k} \int_\Omega \eta \tilde{E}^{(l)} \cdot  \tilde{E}^{(l)} \, dx,
\end{equation} 
for all $\eta \in W^{1,\infty} \left(\Omega\right) \cap  \mathrm{Sym}_3 (\Omega)$, where
\begin{equation*}
c_k :=  \sum_{\substack{0 \leq s_1 \leq |F_1| \\ \ldots \\ 0 \leq s_n \leq |F_n| \\s_1+\ldots +s_n =s}}
\binom{\, \abs{F_k}-1}{s_k-1} (\lambda_{F_k}[\tilde{\varepsilon}])^{s_k} \prod_{\substack{j=1\\j \neq k}}^n
\binom{\, \abs{F_j}}{s_j}(\lambda_{F_j}[\tilde \varepsilon])^{s_j},
\end{equation*}
and for each $k=1, \dots, n$, $\{\tilde E^{(l)}\}_{l \in F_k}$ is an orthonormal basis in $L_{\tilde \varepsilon}^2(\Omega)$ of Maxwell eigenvectors for the eigenspace associated with $\lambda_{F_k}[\tilde \varepsilon]$.
\end{theorem}
\begin{proof}
Let $\tilde \varepsilon \in \mathcal{E}$. As we have already pointed out, Maxwell eigenvalues are independent on the choice of the parameter $\tau>0$ in \eqref{prob:eigen2weak}.
Thus, to avoid  problems of different enumeration  between Maxwell eigenvalues and the eigenvalues of $S_\varepsilon$, we can fix $\tau$ big enough such that all the Maxwell eigenvalues  $\{\lambda_j[\tilde \varepsilon]\}_{j \in F}$ are strictly smaller than  any other eigenvalue of \eqref{prob:eigen2weak} which is not a Maxwell eigenvalue (i.e. an eigenvalue belonging to the family ii) 
in Theorem \ref{thm:codau}). In this way $\sigma_j[\tilde \varepsilon] =\lambda_j[\tilde \varepsilon]$ for all $j \in F$.

The eigenvalues $\mu_j$ of the operator $S_\varepsilon$ and the eigenvalues $\sigma_j$ of \eqref{prob:eigen2weak} satisfy $\mu_j= ({\sigma_j} +1)^{-1}$.   Then the sets $\mathcal{E}[F]$ and $\set{\varepsilon \in \mathcal{E} : \mu_j[\varepsilon] \neq \mu_l[\varepsilon] \,\,\, \forall j \in F, \, l \in \mathbb{N} \setminus F}$ coincide locally around $\tilde \varepsilon$. 
By Lemma \ref{lem:Te}, $S_\varepsilon$ is a compact self-adjoint operator acting on $L^2_\varepsilon(\Omega)$.
Furthermore, as already pointed out in the proof of Lemma \ref{lem:der}, $S_\varepsilon$ depends real analytically on $\varepsilon$. In the same way one shows that also the scalar product $\langle \cdot, \cdot \rangle_\varepsilon$ on $L^2(\Omega)^3$  depends real analytically on $\varepsilon$. 
Therefore, by the abstract result of Lamberti and Lanza de Cristoforis  \cite[Thm. 2.30]{LaLa04}, we have that the set $\set{\varepsilon \in \mathcal{E} : \mu_j[\varepsilon] \neq \mu_l[\varepsilon] \,\,\, \forall j \in F, l \in \mathbb{N} \setminus F}$ is open in $W^{1,\infty} \left(\Omega\right) \cap  \mathrm{Sym}_3 (\Omega)$ and that the function
\[
M_{F,s}[\varepsilon] := \sum_{\substack{j_1,\dots,j_s \in F \\ j_1<\dots<j_s}} \mu_{j_1}[\varepsilon] \cdots \mu_{j_s}[\varepsilon]
\]
depend real analytically on $\varepsilon \in \mathcal{E}[F]$. From this, to infer the real analyticity of the functions $\Lambda_{F,s}$ on $\varepsilon \in \mathcal{E}[F]$, one can just observe that if we denote 
\[
\hat{\Lambda}_{F,s}[\varepsilon] := \sum_{\substack{j_1,\dots,j_s \in F \\ j_1<\dots<j_s}} (\lambda_{j_1}[\varepsilon] +1) \cdots (\lambda_{j_s}[\varepsilon]+1),  
\]
then  we have 
\[
\hat{\Lambda}_{F,s} [\varepsilon] = \frac{M_{F,\, |F|-s}[\varepsilon]}{M_{F,\, |F|}[\varepsilon]}
\]
and  by elementary combinatorics
\begin{equation} \label{Lambdahat:Lambda}
\Lambda_{F,s}[\varepsilon] = \sum_{k=0}^s (-1)^{s-k} \binom{\, \abs{F}-k}{s-k} \hat{\Lambda}_{F,k} [\varepsilon],
\end{equation}
where we have set $\hat{\Lambda}_{F,0}=1$. Then we can deduce that locally around $\tilde \varepsilon$ 
the maps $\Lambda_{F,s}[\varepsilon]$ are real analytic and accordingly the analyticity part of the statement follows since
$\tilde \varepsilon$ is arbitrary.

Next, we turn to prove formula \eqref{diff:Lambdas}. We start by the case $n=1$, that is $F_1 =F$ and 
$\tilde \varepsilon \in \Theta[F]$. Let $\eta \in W^{1,\infty} \left(\Omega\right) \cap  \mathrm{Sym}_3 (\Omega)$.
By  \cite[Thm. 2.30]{LaLa04} we get that
\begin{equation*}
        d\rvert_{\varepsilon=\tilde{\varepsilon}} M_{F,s} [{\eta}] = \binom{\, \abs{F}-1}{s-1} (\lambda_F [\tilde{\varepsilon}]+1)^{1-s} \sum_{l \in F} \langle d\rvert_{\varepsilon=\tilde{\varepsilon}} \, S_{\varepsilon} [\eta][\tilde{E}^{(l)}], \tilde{E}^{(l)} \rangle_{\tilde{\varepsilon}}.
\end{equation*}
Moreover, by using formula \eqref{der:oper} of Lemma \ref{lem:der}, we have that
\begin{equation*}
\begin{split}
& d\rvert_{\varepsilon = \tilde{\varepsilon}} \hat{\Lambda}_{F,s} [\eta] \\
&  =\left( d\rvert_{\varepsilon=\tilde{\varepsilon}} M_{F, \, \abs{F}-s} [\eta] M_{F, \, \abs{F}} \, [\tilde{\varepsilon}] - M_{F, \, \abs{F}-s} \, [\tilde{\varepsilon}] \, d\rvert_{\varepsilon=\tilde{\varepsilon}} M_{F, \, \abs{F}} \,  [\eta] \right) (\lambda_F [\tilde{\varepsilon}]+1)^{2 \, \abs{F}}  \\
& = \left( \binom{\, \abs{F}-1}{\, \abs{F}-s-1} (\lambda_F [\tilde{\varepsilon}] +1)^{s+1-2 \, \abs{F}}- \binom{\, \abs{F}}{s} \binom{\, \abs{F}-1}{\, \abs{F}-1} (\lambda_F [\tilde{\varepsilon}]+1)^{s+1 -2 \, \abs{F}} \right) \\
& \qquad \cdot (\lambda_F [\tilde{\varepsilon}]+1)^{2 \, \abs{F}} \sum_{l \in F} \langle d\rvert_{\varepsilon=\tilde{\varepsilon}} \, S_{\varepsilon} [\eta][\tilde{E}^{(l)}], \tilde{E}^{(l)} \rangle_{\tilde{\varepsilon}}\\
&   = -  \lambda_F [\tilde{\varepsilon}] (\lambda_F [\tilde{\varepsilon}]+1)^{s-1} \binom{\, \abs{F}-1}{s-1} \sum_{l \in F} \int_\Omega \eta \tilde{E}^{(l)} \cdot \tilde{E}^{(l)} \, dx.
\end{split}
\end{equation*}
Finally, recalling \eqref{Lambdahat:Lambda}, we get
\begin{equation*}
\begin{split}
&d\rvert_{\varepsilon = \tilde{\varepsilon}} \Lambda_{F,s} [\eta]\\
&   = -  \lambda_F [\tilde{\varepsilon}] \sum_{k=1}^s (-1)^{s-k}  (\lambda_F [\tilde{\varepsilon}]+1)^{k-1} \binom{\, \abs{F} -k}{s-k} \binom{\, \abs{F}-1}{k-1} \sum_{l \in F} \int_\Omega \eta \tilde{E}^{(l)} \cdot \tilde{E}^{(l)} \, dx\\
&  = -  \binom{\, \abs{F}-1}{s-1} \lambda_F [\tilde{\varepsilon}] \sum_{k=0}^{s-1} \binom{s-1}{k}  (\lambda_F [\tilde{\varepsilon}]+1)^{k} (-1)^{s-k-1} \sum_{l \in F} \int_\Omega \eta \tilde{E}^{(l)} \cdot \tilde{E}^{(l)} \, dx\\
&= -  \binom{\, \abs{F}-1}{s-1} (\lambda_F [\tilde{\varepsilon}])^s \sum_{l \in F} \int_\Omega \eta \tilde{E}^{(l)} \cdot \tilde{E}^{(l)} \, dx.
\end{split}
\end{equation*}
Next we consider the case $n \geq 2$. By means of a continuity argument, one can easily see
that there exists an open neighborhood $\mathcal{W}$ of  $\tilde \varepsilon$ in $\mathcal{E}[F]$ such that $\mathcal{W} \subseteq \bigcap_{k=1}^n\mathcal{E}[F_k]$. Thus
\[
\Lambda_{F,s}[\varepsilon] = \sum_{\substack{0 \leq s_1 \leq |F_1| \\ \ldots \\ 0 \leq s_n \leq |F_n| \\s_1+\ldots +s_n =s}}
\prod_{k=1}^n\Lambda_{F_k,s_k}[\varepsilon] \qquad \forall \varepsilon \in \mathcal{W}.
\]
Differentiating the above equality at the point $\tilde{\varepsilon}$ and using formula \eqref{diff:Lambdas} with $n=1$ to each function $\Lambda_{F_k,s_k}$, one can see that formula \eqref{diff:Lambdas} holds true for any $n \in \mathbb{N}$.
\end{proof}

We conclude this section by studying the case of  one-parametric families of permittivities.
Using \Cref{lem:der} and classical analytic perturbation theory we can recover a Rellich-Nagy-type theorem which allows us to describe all the eigenvalues splitting from a multiple eigenvalue of multiplicity $m$ by means of $m$ real-analytic functions.
For classical results in analytic perturbation theory we refer to the seminal works of Rellich \cite{Re37} and Nagy \cite{Na48}. More up do date formulations  can be found in Chow and Hale \cite[Theorem 5.2, p. 487]{ChHa82}, Kato \cite[Theorem 3.9, p. 393]{Ka95},
Lamberti and Lanza de Cristoforis \cite[Theorem 2.27]{LaLa04}.

\begin{theorem} \label{thm:RN}
Let $\Omega$ be as in \eqref{Omega_def}.
Let $\tilde \varepsilon \in \mathcal{E}$ and let $\{\varepsilon_t\}_{t \in \mathbb{R}} \subseteq  \mathcal{E}$ be a family depending real analytically 
on $t$ and such that $\varepsilon_0 = \tilde \varepsilon$. Let $\tilde \lambda$ be a Maxwell eigenvalue of multiplicity $m \in \mathbb{N}$ and $\tilde{E}^{(1)},\dots, \tilde{E}^{(m)}$  a corresponding orthonormal basis of Maxwell eigenvectors in $L^2_{\tilde \varepsilon}(\Omega)$  with $\varepsilon = \tilde \varepsilon$. 
Let $\tilde \lambda = \lambda_n[\tilde \varepsilon] = \cdots = \lambda_{n+m-1}[\tilde \varepsilon]$ for some $n \in \mathbb{N}$.
Then there exist an open interval $I \subseteq \mathbb{R}$ containing zero and $m$ real analytic functions $g_1,\ldots, g_m$ from 
$I$ to $\mathbb{R}$ such that 
\[
\{\lambda_n[\varepsilon_t] , \ldots , \lambda_{n+m-1}[ \varepsilon_t]\} = \{g_1(t),\ldots, g_m(t)\} \qquad \forall t \in I.
\]
Moreover, the derivatives $g'_1(0), \ldots, g'_m(0)$ of the functions $g_1,\ldots, g_m$ at zero coincide with the eigenvalues of the matrix
\begin{equation*}
\left( -\tilde \lambda \int_\Omega \dot \varepsilon_0 \, \tilde{E}^{(i)} \cdot \tilde{E}^{(j)} \, dx \right)_{i,j=1,\ldots,m},
\end{equation*}
where $\dot \varepsilon_0$ denotes the derivative at $t=0$ of the map $t \mapsto \varepsilon_t$.
\end{theorem}
\begin{proof}
Again, we can assume that $\tau$ is big enough such that $\tilde \lambda$ is strictly smaller than any eigenvalue of \eqref{prob:eigen2weak}  which is not a Maxwell eigenvalue.
By applying \cite[Thm. 2.27, Cor. 2.28]{LaLa04} to the operator $S_\varepsilon$ defined in \eqref{operator:dfn} we get that there exist an open interval $I$ of $\mathbb{R}$ containing zero and $m$ real analytic functions $h_1, \dots, h_m$ from $I$ to $\mathbb{R}$ such that $\{(\lambda_n[\varepsilon_t] +1)^{-1} , \ldots , (\lambda_{n+m-1}[ \varepsilon_t] + 1)^{-1} \} = \{h_1(t),\ldots, h_m(t)\}$ for all $t \in I$. Furthermore, the derivatives at zero of the functions $h_i, i=1, \dots, m$ coincide with the eigenvalues of the matrix
\begin{equation*}
\left( \langle d\rvert_{\varepsilon=\tilde{\varepsilon}} S_\varepsilon [\dot{\varepsilon}_0] \tilde{E}^{(i)}, \tilde{E}^{(j)} \rangle_{\tilde{\varepsilon}} \right)_{i,j=1,\dots,m}.
\end{equation*}
By continuity we have that, eventually further restricting the interval $I$, the functions $h_i$ are away from zero for all $t \in I$. Then, setting 
$$g_i(t):= \frac{1}{h_i(t)} - 1$$
we have that $\{\lambda_n[\varepsilon_t] , \ldots , \lambda_{n+m-1}[ \varepsilon_t]\} = \{g_1(t),\ldots, g_m(t)\}$.
Finally, noticing that 
\begin{equation*}
\frac{d}{dt} g_i(t) \rvert_{t=0} = -(\tilde{\lambda}+1)^2 \frac{d}{dt} h_i(t) \rvert_{t=0},
\end{equation*}
we deduce that the derivatives at zero of the functions $g_i$ coincide with the eigenvalues of the matrix
\begin{equation*}
-(\tilde{\lambda}+1)^2 \left( \langle d\rvert_{\varepsilon=\tilde{\varepsilon}} S_\varepsilon [\dot{\varepsilon}_0] \tilde{E}^{(i)}, \tilde{E}^{(j)} \rangle_{\tilde{\varepsilon}} \right)_{i,j=1,\dots,m} = \left( -\tilde{\lambda} \int_\Omega \dot{\varepsilon}_0 \, \tilde{E}^{(i)} \cdot \tilde{E}^{(j)}\, dx\right)_{i,j=1,\dots,m},
\end{equation*}
where this last equality is justified by \Cref{lem:der}.
\end{proof}

\section{The spectrum is simple for generic permittivities}\label{sec:gen}

The issue of understanding if the eigenvalues of a parameter dependent problem can be made all simple by an 
arbitrarily small perturbation of the parameter is a natural question and has been already investigated by several authors for different problems. For example,  Albert \cite{Al75}  proved the generic simplicity of 
the spectrum of an elliptic operator with respect to the perturbation of the zeroth order term. Moreover, the generic simplicity of the spectrum
has been also considered with respect to the domain perturbation in various papers. We mention, e.g, Micheletti 
\cite{Mi72, Mi73} for the Laplacian and for a general elliptic operator and Ortega and Zuazua \cite{OrZu01} and 
Chitour, Kateb and Long \cite{ChKaLo16} for the Stokes system in dimension two and three, respectively. Finally, we also mention the more recent paper by Dabrowski \cite{Da21} where the author analyze the Laplacian with different boundary conditions and consider also singular perturbations of the domain.

A first step, as we will show in the next proposition, is to prove that  it is always possible to find a small perturbation of the permittivity  that  splits  a non-zero Maxwell eigenvalue of multiplicity $m$ into $m$ simple eigenvalues. 

\begin{proposition} \label{first:step:genericity}
Let $\Omega$ be as in \eqref{Omega_def}.  Let $\tilde \varepsilon \in \mathcal{E}$, $\tilde \lambda \neq 0$ a Maxwell eigenvalue of multiplicity $m \in \mathbb{N}$ and $\tilde{E}^{(1)},\ldots, \tilde{E}^{(m)}$  a corresponding orthonormal basis of Maxwell eigenvectors in $L^2_{\tilde \varepsilon}(\Omega)$ with $\varepsilon = \tilde \varepsilon$. 
Let $\tilde \lambda = \lambda_n[\tilde \varepsilon] = \cdots = \lambda_{n+m-1}[\tilde \varepsilon]$ for some $n \in \mathbb{N}$.
Define 
\[
\tilde \varepsilon_{t,\eta} := \tilde \varepsilon +t \eta \qquad \forall t \in \mathbb{R},
\]
for all $\eta \in W^{1,\infty} \left(\Omega\right) \cap  \mathrm{Sym}_3 (\Omega)$, $\norm{\eta}_{ W^{1,\infty}(\Omega)}\leq 1$.  
Then for all $T>0$ there exist $\eta \in W^{1,\infty} \left(\Omega\right) \cap  \mathrm{Sym}_3 (\Omega)$ with $\norm{\eta}_{ W^{1,\infty}(\Omega)}\leq 1$, and $t\in \mathopen]0,T[$ such that $\tilde \varepsilon_{t,\eta}  \in \mathcal{E}$ and the eigenvalues $\lambda_n[\tilde \varepsilon_{t,\eta}] , \ldots , \lambda_{n+m-1}[\tilde  \varepsilon_{t,\eta}]$ are all simple. 
\end{proposition}
\begin{proof}
We will only prove that there exist $\eta \in  W^{1,\infty} \left(\Omega\right) \cap  \mathrm{Sym}_3 (\Omega)$ with $\norm{\eta}_{W^{1,\infty}(\Omega)} \leq 1$ and $t>0$  as small as desired such that the eigenvalues $\lambda_n[\tilde \varepsilon_{t,\eta}] , \ldots , \lambda_{n+m-1}[\tilde  \varepsilon_{t,\eta}]$  are not all equal. Then, repeating the same argument for the eigenvalues that have still a multiplicity strictly greater than one, in a finite number of steps we are done.
Note that by the  continuity of the eigenvalues with respect to permittivity variations and by choosing $t$ small enough we can avoid that the eigenvalues splitting from a multiple eigenvalue could overlap or switch position with other eigenvalues.

Hence, suppose by contradiction that there exists $T >0$ such that for all $\eta \in W^{1,\infty} \left(\Omega\right) \cap  \mathrm{Sym}_3 (\Omega)$ with $\norm{\eta}_{W^{1,\infty}(\Omega)} \leq 1$ and for all 
$t \in \mathopen]0,T[$,   all the eigenvalues  
$\lambda_n[\tilde \varepsilon_{t,\eta}] , \ldots , \lambda_{n+m-1}[\tilde  \varepsilon_{t,\eta}]$  coincide. As a consequence, all the right derivatives  at $t = 0$ 
of the branches coincide. Then, if we fix $\eta$ and use \Cref{thm:RN}, we get that all the eigenvalues of the matrix
\begin{equation} \label{matrix:rellichnagy:dfnM}
M:=\left( - \tilde \lambda \int_\Omega \eta \tilde{E}^{(i)} \cdot \tilde{E}^{(j)} \, dx \right)_{i,j=1,\ldots,m}
\end{equation}
coincide. Since the above matrix is a real symmetric matrix with only one eigenvalue, it is a scalar matrix. In other words, there exists $\mu[\eta] \in \mathbb{R}$ such that 
\begin{equation}\label{eq:M}
M = \mu[\eta] \, I_m,
\end{equation}
where $I_m$ denotes the $(m \times m)$-identity matrix.
For $h=1,2,3$ we set
\[
\eta_h := \norm{\xi}_{W^{1,\infty}(\Omega)}^{-1} \xi \, e_{hh}
\]
with $0\neq \xi \in C_c^1(\Omega)$ arbitrary and $e_{hh}$ the $(3 \times 3)$-matrix with $(h,h)$-entry equal to $1$ and zeros elsewhere. Since $\tilde \lambda \neq 0$, by \eqref{matrix:rellichnagy:dfnM}, \eqref{eq:M} and using the above defined $\eta_h$ we can recover that for all $\xi \in C_c^1(\Omega)$
\[
\int_{\Omega}\xi \, E^{(i)}_h E^{(j)}_h \,dx=0 \qquad \forall i,j \in \{1,\ldots,m\}, i \neq j, \quad \forall h=1,2,3,
\]
and 
\[
\int_{\Omega}\xi \left( (E^{(i)}_h)^2- (E^{(j)}_h)^2 \right)\,dx= 0 \qquad \forall i,j \in \{1,\ldots,m\}, \quad \forall h=1,2,3.
\]
By the fundamental lemma of calculus of variations we get that a.e. in $\Omega$
\[
E^{(i)}_h E^{(j)}_h=0    \qquad  \forall i,j \in \{1,\ldots,m\},\, i \neq j, \quad \forall h=1,2,3,
\]
and 
\[
(E^{(i)}_h)^2- (E^{(j)}_h)^2 =0  \qquad  \forall i,j \in \{1,\ldots,m\}, \quad \forall h=1,2,3.
\]
The above relations clearly implies that $E_i=0$ for all $ i \in \{1,\ldots,m\}$, which is a contradiction since they are not identically zero, being eigenfunctions. 
\end{proof}

\begin{remark}
The constraint $\norm{\eta}_{ W^{1,\infty}(\Omega)}\leq 1$ in the above proposition can be replaced by $\norm{\eta}_{ W^{1,\infty}(\Omega)}\leq \delta$ for any $\delta >0$.
\end{remark}

\begin{remark}
The argument we have used to split a multiple eigenvalue into several eigenvalues of lower multiplicity uses that $\eta$ is a general symmetric matrix and not a scalar matrix. However, noticing in which way $\eta_h$ is defined, one can easily realize that such an argument still works if $\eta$ varies in the class of diagonal matrices. Instead, in the case that we restrict ourselves to the case of scalar matrices, what we can recover by arguing in the same way is that 
\[
E^{(i)} \cdot E^{(j)}=0   \qquad  \forall i,j \in \{1,\ldots,m\}, i \neq j.
\]
and 
\[
|E^{(i)}|^2- |E^{(j)}|^2 = 0 \qquad  \forall i,j \in \{1,\ldots,m\}.
\]
This does not immediately lead to a contradiction.  Thus, it would be interesting to investigate whether it is still possible 
to split the whole spectrum when the permittivies are scalar.
\end{remark}

We are now ready to show that the whole positive Maxwell spectrum is generically simple with respect to the permittivity. We note that our proof is inspired by the methods of Albert \cite{Al75} 
\begin{theorem}\label{thm:gensim}
Let $\Omega$ be as in \eqref{Omega_def}. Let $\tilde{\varepsilon} \in \mathcal{E}$ and let $\delta>0$ be small enough such that 
\[
\tilde \varepsilon + \eta  \in \mathcal{E}
\] 
for all $\eta \in W^{1,\infty} \left(\Omega\right) \cap  \mathrm{Sym}_3 (\Omega)$ with $\|\eta \|_{W^{1,\infty}(\Omega)} \leq \delta$. 
Let
\[
B_0 := \left\{ \eta \in W^{1,\infty} \left(\Omega\right) \cap  \mathrm{Sym}_3 (\Omega): \|\eta \|_{W^{1,\infty}(\Omega)} \leq \delta \right\}
\]
and
\[
B_n :=  \left\{\eta \in B_0: \mbox{ the first $n$ positive Maxwell eigenvalues with $\varepsilon =\tilde \varepsilon +\eta$ are simple}\right\}
\]
for $n \in \mathbb{N}$.
Then
\[
B:= \bigcap_{n \in \mathbb{N}}B_n = \left\{ \eta \in B_0: \mbox{ all the positive Maxwell eigenvalues  with $\varepsilon =\tilde \varepsilon +\eta$ are simple}\right\}
\]
is dense in $B_0$. 
\end{theorem}
\begin{proof}
The proof follows by applying the Baire's lemma in the complete metric space $B_0$. In order to do this, we have to show that
\begin{itemize}
\item[i)] $B_n$ is open in $B_0$ for all $n \in \mathbb{N}$,
\item[ii)] $B_{n+1}$ is dense in $B_n$ for all $n \in \mathbb{N}$.
\end{itemize}
Statement i) follows from the continuity of the eigenvalues with respect to the permittivity parameter (see Theorem \ref{thm:loclip}). Next we prove statement ii) by contradiction. 
Assume that $B_{n+1}$ is not dense in $B_n$ for some $n \in \mathbb{N}$. Then there exists $\eta \in B_n \setminus B_{n+1}$ and 
a neighborhood $U$ of $\eta$ in $B_0$ such that 
\[
U \subseteq B_n \setminus B_{n+1}.
\]
Since $\eta \in B_n\setminus B_{n+1}$ then
\begin{itemize}
\item the first $n$ non-zero Maxwell eigenvalues with $\varepsilon=\tilde \varepsilon +\eta$ are simple,
\item  the $(n+1)$-th non-zero Maxwell eigenvalue with $\varepsilon=\tilde \varepsilon +\eta$ has multiplicity $k$ for some $k \in \mathbb{N}$, $k \geq 2$.
\end{itemize}
Moreover, we note that for all $\rho \in U \subseteq B_n \setminus B_{n+1}$ we have:
\begin{itemize}
\item the first $n$ non-zero Maxwell eigenvalues with $\varepsilon=\tilde \varepsilon +\rho$ are simple,
\item the $(n+1)$-th  non-zero Maxwell eigenvalue with  $\varepsilon=\tilde \varepsilon +\rho$  is not simple.
\end{itemize}
By \Cref{first:step:genericity} there exist $\hat\rho \in W^{1,\infty} \left(\Omega\right) \cap  \mathrm{Sym}_3 (\Omega)$ with $\|\hat\rho\|_{W^{1,\infty}(\Omega) }\leq 1$ and $t>0$ arbitrarily small such that $\eta + t\hat \rho \in U$ and  all the non-zero Maxwell eigenvalues with $\varepsilon=\tilde \varepsilon + \eta + t\hat \rho$ with indices from $(n+1)$ to $(n+k)$ are simple, therefore we deduce that in particular $\eta + t\hat \rho \in B_{n+1}$. This 
is a contradiction since 
$U \subseteq B_n \setminus B_{n+1}$.
\end{proof}

\subsection*{Acknowledgment}

The authors are members of the `Gruppo Nazionale per l'Analisi Matematica, la Probabilit\`a e le loro Applicazioni' (GNAMPA) of the `Istituto Nazionale di Alta Matematica' (INdAM) and acknowledge the support of the Project BIRD191739/19 `Sensitivity analysis of partial differential equations in
the mathematical theory of electromagnetism' of the University of Padova.
The second author was partially supported by `Fondazione Ing. Aldo Gini' during the preparation of this paper. 
The authors are deeply thankful to Prof. Pier Domenico Lamberti  for many valuable comments during the preparation of the paper.

\end{document}